\documentclass[english, 10pt]{amsart}

\usepackage{babel}
\usepackage{amstext}
\usepackage{amsmath}
\usepackage{amssymb}
\usepackage{amsfonts}
\usepackage{latexsym} 
\usepackage{ifthen}
\usepackage{xypic}
\xyoption{all}
\usepackage{amsmath,amssymb,amsfonts}
\usepackage{textcomp}
\usepackage{amsmath,amscd}
\usepackage{pictexwd,dcpic}

\usepackage{enumerate}
\usepackage{verbatim}

\renewcommand{\O}{{\mathcal O}}

\newcommand{\rk}{{\rm rk}}

\newcommand{\Bs}{{\rm Bs}}

\newcommand{\sJ}{{\mathcal J}}

\newtheorem{lemma1}{}[section]

\newenvironment{lemma}{\begin{lemma1}{\bf Lemma.}}{\end{lemma1}}
\newenvironment{example}{\begin{lemma1}{\bf Example.}\rm}{\end{lemma1}}

\newenvironment{theorem}{\begin{lemma1}{\bf Theorem.}}{\end{lemma1}}
\newenvironment{proposition}{\begin{lemma1}{\bf Proposition.}}{\end{lemma1}}
\newenvironment{corollary}{\begin{lemma1}{\bf Corollary.}}{\end{lemma1}}
\newenvironment{remark}{\begin{lemma1}{\bf Remark.}\rm}{\end{lemma1}}
\newenvironment{definition}{\begin{lemma1}{\bf Definition.}}{\end{lemma1}}

\newenvironment{setup}{\begin{lemma1}{\bf Setup.}}{\end{lemma1}}

\newenvironment{conjecture}{\begin {lemma1}{\bf Conjecture.}}{\end{lemma1}}

\newenvironment{remark*}{{\bf Remark.}}{}
\newenvironment{example*}{{\bf Example.}}{}

\newcommand{\R}{\ensuremath{\mathbb{R}}}
\newcommand{\Q}{\ensuremath{\mathbb{Q}}}
\newcommand{\Z}{\ensuremath{\mathbb{Z}}}

\newcommand{\N}{\ensuremath{\mathbb{N}}}
\newcommand{\PP}{\ensuremath{\mathbb{P}}}

\newcommand{\merom}[3]{\ensuremath{#1\colon #2 \dashrightarrow #3}}

\newcommand{\holom}[3]{\ensuremath{#1\colon #2  \rightarrow #3}}
\newcommand{\fibre}[2]{\ensuremath{#1^{-1} (#2)}}

\makeatletter
\ifnum\@ptsize=0  \fi \ifnum\@ptsize=2  \fi 

\newcommand\sE{{\mathcal E}}

\newcommand\sF{{\mathcal F}}
\newcommand\sG{{\mathcal G}}

\newcommand\sI{{\mathcal I}}

\newcommand\sO{{\mathcal O}}

\newcommand\Hom{{\rm Hom}}

\DeclareMathOperator{\nd}{nd}

\DeclareMathOperator{\Id}{Id}

\newcommand{\NE}[1]{ \ensuremath{ \overline { \mbox{NE} }(#1)} }

\title{Manifolds with nef anticanonical bundle} 
\date{May 5, 2013}

\author{Junyan Cao}
\author{Andreas H\"oring}


\address{Junyan Cao, 
Universit{\'e} de Grenoble I, Institut Fourier, 38402 Saint-Martin d'H\`{e}res, France}
\email{junyan.cao@ujf-grenoble.fr}

\address{Andreas H\"oring, Universit{\'e} Pierre et Marie Curie, Institut de math{\'e}matiques de Jussieu,
Projet Topologie et g{\'e}om{\'e}trie alg{\'e}briques, Case 247,  4 place Jussieu, 75005 Paris, France}
\email{hoering@math.jussieu.fr}

\begin{document}

\begin{abstract}
Let $X$ be a compact K\"ahler manifold such that the anticanonical bundle $-K_X$ is nef.
A classical conjecture claims that the Albanese map $X \rightarrow T$ is submersive.
We prove this conjecture if the general fibre is a weak Fano manifold. If $X$ is projective
we prove the conjecture also for fibres of dimension at most two.
\end{abstract}

\maketitle

\section{Introduction}

Let $X$ be a compact K\"ahler manifold such that the anticanonical bundle $-K_X$ is nef, and 
let \holom{\pi}{X}{T} be the Albanese map. By the work of Zhang \cite{Zha96}, P\v aun \cite{Pau12}
and the first named author \cite{Cao13} we know that $\pi$ is a fibration, i.e. $\pi$ is surjective and has connected fibres.
The aim of this paper is to give evidence for the following:

\begin{conjecture} \cite{DPS94} \label{conjecturealbanese}
Let $X$ be a compact K\"ahler manifold such that  $-K_X$ is nef, and 
let \holom{\pi}{X}{T} be the Albanese map.
Then the fibration $\pi$ is smooth. 
If the general $\pi$-fibre is simply connected, the fibration $\pi$ is locally trivial in the analytic topology.
\end{conjecture}

This conjecture has been proven under the 
stronger assumption that $T_X$ is nef or $-K_X$ is hermitian 
semipositive \cite{CP91, DPS93, DPS94, DPS96, CDP12}, but the general case is 
very much open: so far it is only known in the case where $q(X)=\dim X$, i.e. the Albanese map is birational \cite{Zha96, Fan06}. If $X$ is projective we also know that $\pi$ is equidimensional and has reduced fibres \cite{LTZZ10}.
In low dimension explicit computations based on the minimal model program (MMP) allow to say more:

\begin{theorem} \label{theoremPS} \cite[Thm.]{PS98} 
Let $X$ be a projective manifold such that  $-K_X$ is nef, and 
let \holom{\pi}{X}{T} be the Albanese map. If $\dim X \leq 3$, then $\pi$ is smooth.
\end{theorem}

We prove Conjecture \ref{conjecturealbanese} when the general fibre is a weak Fano manifold:

\begin{theorem} \label{theoremmain}
Let $X$ be a compact K\"ahler manifold such that  $-K_X$ is nef, and 
let \holom{\pi}{X}{T} be the Albanese map.
Let $F$ be a general $\pi$-fibre. If $-K_F$ is nef and big, then $\pi$ is locally trivial
in the analytic topology.
\end{theorem}

Let us explain the strategy of proof under the stronger assumption that $-K_X$ is $\pi$-ample:
for $m \gg 0$ we have an embedding 
$$
X \hookrightarrow \PP(\pi_* (\omega_X^{\otimes -m})).
$$
The main technical point is to 
show that the direct image sheaf $\pi_* (\omega_X^{\otimes -m})$ is a nef vector bundle and $(-K_X)^{\dim X-\dim T+1}=0$.
If $X$ is projective this is not very difficult, the non-algebraic case needs substantially more effort and should be
of independent interest.  
Combining these two facts an intersection computation shows that $\pi_* (\omega_X^{\otimes -m})$ is actually numerically flat, i.e.
nef and antinef. Yet a numerically flat vector bundle is a rather special 
local system \cite[Sect.3]{Sim92}, so an argument from \cite{Cao12} 
allows to show that the equations of the fibres $X_t \subset  \PP(\pi_* (\omega_X^{\otimes -m})_t)$ do 
not depend on $t \in T$. In particular all the fibres are isomorphic, so $\pi$ is locally trivial.
If $-K_X$ is only nef and $\pi$-big, the same considerations show that the relative anticanonical fibration $X' \rightarrow T$
is locally trivial. We then use birational geometry to deduce that $X \rightarrow T$ is also locally trivial.
Theorem \ref{theoremmain} immediately implies:

\begin{corollary} \label{corollarymain}
Let $X$ be a compact K\"ahler manifold such that $-K_X$ is nef. 
Conjecture \ref{conjecturealbanese} holds if $q(X) = \dim X-1$. 
\end{corollary}

This also settles the problem in low dimension.

\begin{corollary} \label{corollarykaehler}
Let $X$ be a compact K\"ahler manifold such that $-K_X$ is nef. 
Conjecture \ref{conjecturealbanese} holds if $\dim X \leq 3$.
\end{corollary}

In the second part of the paper we turn our attention to the case where the positivity of $-K_X$ is not strict, even
along the general $\pi$-fibre. We use the MMP to prove Conjecture \ref{conjecturealbanese} for fibres of low dimension.

\begin{theorem} \label{theoremmaintwo}
Let $X$ be a projective manifold such that  $-K_X$ is nef.
Conjecture \ref{conjecturealbanese} holds if $q(X) = \dim X-2$. 
\end{theorem}

The basic idea of the proof is very simple: find a Mori contraction $\mu: X \rightarrow Y$ onto 
a projective manifold $Y \rightarrow T$ such that $-K_{Y}$ is nef and relatively big.
Then $-K_X-\mu^* K_Y$ is nef and relatively big, using the birational morphism
$$
X \rightarrow X' \subset \PP(\pi_* (\omega_X^{\otimes -m} \otimes \mu^* \omega_Y^{\otimes -m})).
$$
we can prove as in Theorem \ref{theoremmain} that $X \rightarrow T$ is locally trivial.
Unfortunately it is a priori not clear that such a contraction $X \rightarrow Y$ exists.
In fact the second named author constructed an example of a 
(rationally connected) projective threefold $M$ such that $-K_M$ is nef and not big
and $M=\mbox{Bl}_B M'$ with $B$ a smooth rational curve such that $-K_{M'} \cdot B<0$ \cite[Ex.4.8]{a8}. 
This problem already appeared in the work of Peternell and Serrano and we follow
the same strategy to overcome this difficulty: let $X'/T$ be a Mori fibre space birational to $X/T$,
then try to prove that $-K_{X'}$ is nef. Once this property is established one can
describe precisely all the steps of the MMP $X \dashrightarrow X'$.
The contribution of this paper is to introduce a new method to establish this kind of statement:
our proof is based on the idea that if we restrict the MMP to some (pluri-)anticanonical divisor $D' \subset X'$,
the numerical dimension of $-K_{X'}|_{D'}$ is zero or one. This observation quickly leads to strong
restrictions on the MMP in a neighbourhood of $D'$, cf. Lemma \ref{lemmanu2}.
The main point is thus to show the existence of global sections of $-m K_X$ for some $m \in \N$: 
this can be done on threefolds, but is completely open in higher dimension.

{\bf Acknowledgements.} We thank I.Biswas, S.Boucksom, T.Dinh, V. Lazi\'c, W.Ou, T. Sano and C.Simpson for helpful communications.
We thank J-P. Demailly for helpful discussions and numerous suggestions.
A. H\"oring was partially supported by the A.N.R. project CLASS\footnote{ANR-10-JCJC-0111}.

\begin{center}
{\bf
Notation and terminology
}
\end{center}

For general definitions - at least in the algebraic context -  we refer to Hartshorne's book \cite{Har77}.
We will frequently use standard terminology and results 
of the minimal model program (MMP) as explained in \cite{KM98} or \cite{Deb01}.
Manifolds and varieties will always be supposed to be irreducible.
A fibration is a proper surjective map with connected fibres \holom{\varphi}{X}{Y} between normal varieties.

Let us recall the various positivity concepts that will be used in this paper.

\begin{definition} \label{definitionnef} \cite{Dem12}
Let $(X, \omega_{X} )$ be a compact K\"ahler manifold, and let $\alpha \in H^{1,1}(X) \cap H^2(X, \R)$ be a real cohomology class 
of type $(1,1)$. We say that $\alpha$ is nef if for every $\epsilon> 0$, there is a smooth $(1,1)$-form $\alpha_{\epsilon}$
in the same class of $\alpha$ such that $\alpha_{\epsilon}\geq -\epsilon\omega_{X}$.

We say that $\alpha$ is pseudoeffective if there exists a $(1, 1)$-current $T\geq 0$ in the same class of $\alpha$.
We say that $\alpha$ is big if there exists a $\epsilon> 0$ such that $\alpha-\epsilon \omega_{X}$ is pseudoeffective.
\end{definition}

\begin{definition} \label{definitionrelativebig}
Let $\alpha$ be a nef class on a compact K\"ahler manifold $X$, and let $\pi: X\rightarrow T$ be a fibration.
We say that $\alpha$ is $\pi$-big if for a general fibre $F$, the restriction $\alpha|_F$ is big.
\end{definition}

\begin{definition} \cite[Def 6.20]{Dem12} \label{definitionnumericaldimension}
Let $X$ be a compact K\"ahler manifold, and let $\alpha \in H^{1,1}(X) \cap H^2(X, \R)$ be a real cohomology class 
of type $(1,1)$. Suppose that $\alpha$ is nef.
We define the numerical dimension of $\alpha$ by 
$$
\nd (\alpha) :=
\max \{k \in \N \ | \ \alpha^{k}\neq 0 \mbox{ in } H^{2k}(X,\mathbb{R})\}.
$$
\end{definition}

\begin{remark} \label{remarknumericaldimension} In the situation above, set $m=\nd (\alpha)$.
By \cite[Prop 6.21]{Dem12} the cohomology class $\alpha^{m}$ can be represented 
by a non-zero closed positive $(m,m)$-current $T$.
Therefore 
$\int_X \alpha^{m}\wedge\omega_{X}^{\dim X - m}\neq 0$ for any K\"ahler class $\omega_{X}$.
\end{remark}

\begin{definition} \label{definitionnefcodimone}
Let $M$ be a projective variety, and let $L$ be a $\Q$-Cartier divisor on $M$. We say that $L$
is nef in codimension one if $L$ is pseudoeffective and for every prime divisor $D \subset M$,
the restriction $L|_D$ is pseudoeffective.
\end{definition}

\begin{remark} \label{remarknefcodimone}
If $M$ is a normal projective variety and $L$ a $\Q$-Cartier divisor which is nef 
in codimension one, then $L^2$ is a pseudoeffective cycle, i.e. a limit of effective cycles of codimension two.
Indeed if 
$L= \sum \lambda_j D_j +N$ is the divisorial Zariski decomposition \cite{Bou04, Nak04}, we have
$$
L^2 = \sum \lambda_j L|_{D_j} + L \cdot N.
$$
By hypothesis the restriction $L|_{D_j}$ is pseudoeffective, so a limit of effective divisors on $D_j$.
The class $N$ is modified nef in the sense of \cite{Bou04}, so its intersection with any pseudoeffective divisor
gives a pseudoeffective cycle. 
\end{remark}

\begin{definition} \cite{Miy87} \label{definitiongenericallynef}
Let $X$ be a normal, projective variety of dimension $n$, and let 
$\sF$   be a torsion free coherent sheaf on $X$. We say that $\sF$ is generically nef with
respect to a polarisation $A$ on $X$ 
if $\sF|_C$ is nef where
\[
C := D_1 \cap \ldots \cap D_{n-1}
\]
with $D_j \in | m_j A |$ general and $m_j \gg 0$. 
\end{definition}

\section{Numerical dimension} \label{sectionnumericaldimension}

In this section we give an upper bound for the numerical dimension of $-K_X$:

\begin{proposition} \label{propositionnumericaldimension}
Let $X$ be a compact K\"ahler manifold of dimension $n$ such that  $-K_X$ is nef, and 
let \holom{\pi}{X}{T} be the Albanese map. Set $r:=\dim T$.
If $-K_X$ is $\pi$-big, we have $\nd (-K_X)=n-r$.
\end{proposition}

\begin{remark} \label{remarkprojective} If the torus $T$ is projective, this statement is well-known: in this case the manifold $X$ is 
also projective, so
if $\nd(-K_X)>n-r$ we can apply Kawamata-Viehweg vanishing \cite[6.13]{Dem00} to see that 
\begin{equation}\label{projkawaview}
H^{r}(X, \sO_X) = H^{r}(X, K_X+(-K_X))=0.
\end{equation}
The pull-back of a non-zero holomorphic $r$-form from $T$ gives an immediate contradiction. 
We will also use the following special case of  \cite[Thm.5.1]{AD11}:
\end{remark}

\begin{lemma} \label{lemmanumericaldimensionprojective}
Let $X$ be a normal projective variety and $\Delta$ a boundary divisor on $X$ such that
the pair $(X, \Delta)$ is klt. Let \holom{\varphi}{X}{C} be a fibration onto a smooth curve such 
that $-(K_{X/C}+\Delta)$ is nef and $\varphi$-big. 
Then we have
$$
(K_{X/C}+\Delta)^{\dim X} = 0.
$$
\end{lemma}
\vspace{5pt}

To prove Proposition \ref{propositionnumericaldimension} for arbitrary K\"ahler manifolds,
we first prove that if $\nd (-K_X)\geq n-r+1$, then $\pi_* ((-K_X)^{n-r+1})$ is nontrivial .
More precisely, we have

\begin{lemma} \label{lemmanumericaldimension}
Let $X$ be a compact K\"ahler manifold of dimension $n$, and let
$\pi: X\rightarrow T$ be a surjective morphism onto a compact K\"ahler manifold $(T, \omega_T)$ of dimension $r$.
Let $L$ be a line bundle on $X$ that if nef and $\pi$-big.
If $\nd(L)\geq n-r+1$, we have
$$
\int_{X}L^{n-r+1}\wedge(\pi^{*}\omega_{T})^{r-1} > 0.
$$
\end{lemma}

\begin{proof}
We suppose that $\nd (L)=n-r+k$ for some $k \in \N^*$.
Since $L$ is nef and $\pi$-big, the class  
\begin{equation} \label{equationone}
\alpha=L+C\cdot\pi^{*}(\omega_{T})
\end{equation}
is a nef class for any fixed constant $C>0$, and
$\int_{X}\alpha^{n} > 0$.
Thanks to \cite[Thm. 0.5]{DP04},
there exists $\epsilon> 0$, such that $\alpha-\epsilon\omega_{X}$ is a pseudoeffective class.
Combining this with the fact that $L$ is nef,
we have
\begin{equation}\label{degrelast}
\int_{X} L^{n-r+k}\wedge \alpha^{r-k}
\geq \epsilon\int_{X} L^{n-r+k}\wedge \alpha^{r-k-1}\wedge\omega_{X}
\end{equation}
$$
\geq \epsilon^2\int_{X} L^{n-r+k}\wedge \alpha^{r-k-2}\wedge\omega^{2}_{X}\geq
\cdots\geq \epsilon^{r-k}\int_{X}L^{n-r+k}\wedge\omega_{X}^{r-k}> 0,$$
where the last inequality comes from Remark \ref{remarknumericaldimension}.
By the definition of numerical dimension and $\eqref{equationone}$, we have
\begin{equation}\label{degrefirst}
C^{n-k}\cdot\int_{X} L^{n-r+k}\wedge (\pi^{*}\omega_{T})^{r-k} = \int_{X} L^{n-r+k}\wedge \alpha^{r-k}.
\end{equation}
Now \eqref{degrelast} and \eqref{degrefirst} imply that
\begin{equation} \label{equationtwo}
\int_{X} L^{n-r+k}\wedge (\pi^{*}\omega_{T})^{r-k}  > 0.
\end{equation}
On the other hand, since $L$ is $\pi$-big, we have
\begin{equation}\label{equationthree}
\int_{X} L^{n-r}\wedge (\pi^{*}\omega_{T})^{r}> 0.
\end{equation}
Using the Hovanskii-Teissier inequality in the K\"{a}hler case (cf. Appendix \ref{appendixinequality}),
the inequalities $\eqref{equationtwo}$ and $\eqref{equationthree}$ imply
$\int_{X}L^{n-r+1}\wedge(\pi^{*}\omega_{T})^{r-1}> 0$.
\end{proof}

We recall a vanishing theorem proved in \cite[Prop. 2.4]{Cao12}

\begin{lemma}\label{keyvanishing1}
Let $L$ be a line bundle on a compact K\"{a}hler manifold $(X,\omega)$ of dimension $n$, and 
let $\varphi$ be a metric on $L$ with analytic singularities. 
Let 
$ \lambda_{1}(x)\leq \lambda_{2}(x)\leq\cdots\leq\lambda_{n}(x)$
be the eigenvalues of $\frac{i}{2\pi}\Theta_{\varphi}(L)$ with respect to $\omega$.
If
\begin{equation}\label{equation1}
\sum_{i=1}^{p} \lambda_{i}(x)\geq c
\end{equation}
for some constant $c> 0$ independent of $x \in X$,
then
$$H^{q}(X, K_{X}\otimes L\otimes \mathcal{I}(\varphi))=0 
\qquad \forall \ q\geq p.$$
\end{lemma}

The following vanishing property plays an important role in the proof of Proposition \ref{propositionnumericaldimension}, 
Although it was essentially proved in \cite[Prop.5.3]{Cao12}, we give the proof since
the situation here is a little bit more general.

\begin{proposition}\label{lemmavanishing}
Let $(X, \omega_{X})$ be a compact K\"ahler manifold of dimension $n$, and
$L$ be a nef line bundle on $X$.
Suppose that the following holds:
\begin{enumerate}[(i)]
\item $X$ admits
a two steps tower fibration 
$$\begin{CD} X @>\pi >> T @> \pi_{1}>> S \end{CD}$$ 
where $\pi$ is a fibration onto a compact K\"ahler manifold $(T, \omega_T)$ of dimension $r$,
and $\pi_{1}$ is a smooth fibration onto a smooth curve $S$.
\item $L$ is $\pi$-big and satisfies
$$\pi_{*}(c_{1}(L)^{n-r+1})=\pi_{1}^{*}(\omega_{S})$$ 
for a K\"ahler metric $\omega_{S}$ on $S$.
\end{enumerate}
Then we have
$$ 
H^{q}(X,K_{X}\otimes L)=0 \qquad \forall \  q\geq r.
$$
\end{proposition}

\begin{proof}
By \cite[Lemma 5.1]{Cao12}\footnote{Note that the proof of \cite[Lemma 5.1]{Cao12} works well in our case.} the class
$L-d\cdot \pi^* \pi_{1}^* \omega_{S}$
is pseudoeffective for some $d> 0$.
Therefore there exists a singular metric $h_{1}$ on $L$
such that 
$$i\Theta_{h_{1}}(L)\geq d\cdot\pi^{*}\pi_{1}^*\omega_{S}.$$
Since $c_{1}(L)+ \pi^{*} \omega_{T}$ is nef and  
$\int_{X}(c_{1}(L)+\pi^{*} \omega_{T})^{n}> 0$,
\cite[Thm. 0.5]{DP04}
implies the existence of a singular metric $h_{2}$ on $L$
such that
$$i\Theta_{h_{2}}(L)\geq c\cdot \omega_{X}- \pi^{*} \omega_{T}$$
in the sense of currents for some constant $c > 0$.
Thanks to a standard regularization theorem \cite{Dem92},
we can suppose moreover that $h_{1}, h_{2}$ have analytic singularities.
Since $L$ is nef we know that for any $\epsilon> 0$, there exists a smooth metric $h_{\epsilon}$ on $L$
such that $i\Theta_{h_{\epsilon}}(L)\geq -\epsilon\omega_{X}$.
Now we define a new metric $h$ on $L$:
$$h=\epsilon_{1} h_{1}+\epsilon_{2} h_{2}+ (1-\epsilon_{1}-\epsilon_{2})h_{\epsilon}$$
for some $1\gg \epsilon_{1}\gg \epsilon_{2} \gg \epsilon> 0$.
By construction, we have
\begin{equation}\label{equation2}
i\Theta_{h}(L)=
\epsilon_{1}i\Theta_{h_{1}}(L)+\epsilon_{2}i\Theta_{h_{2}}(L)+(1-\epsilon_{1}-\epsilon_{2})i\Theta_{h_{\epsilon}}(L)
\end{equation}
$$
\geq d\cdot\epsilon_{1} \pi^{*}(\omega_{S})-\epsilon_{2}\pi^{*}(\omega_{T})+(c \cdot\epsilon_{2}-\epsilon)\omega_{X}.
$$
Set
$\omega_{\tau}=\tau\cdot\omega_{X}+\pi^{*}(\omega_{T})$ for some $\tau> 0$.

We now check that $( i\Theta_{h}(L), \omega_{\tau} )$ satisfies the condition \eqref{equation1} in Lemma \ref{keyvanishing1} 
for $p=r$ when $\tau$ is small enough (i.e., we consider the eigenvalues of $i\Theta_{h}(L)$ with respect to $\omega_{\tau}$,
where $\tau\ll \epsilon$).
Let $x\in X$ and let $V$ be a $r$ dimensional subspace of $(T_X)_x$.
By an elementary estimate, we 
have\footnote{In fact, since $\pi_1$ is a submersion, $\omega_T$ decomposes the tangent bundle of $T$ as $T_{T/S}\oplus \pi_1^* (T_S)$ 
in the sense of $C^{\infty}$. Observing that $\pi_1$ is smooth and $\epsilon_1\gg \epsilon_2$,
we have
\begin{equation}\label{addremark}
d\cdot\epsilon_{1} \pi_1^{*}(\omega_{S})(t,t)-\epsilon_{2}\omega_{T}(t, t)\geq \frac{d\cdot\epsilon_{1}}{2}\omega_{T}(t,t)
\qquad\text{for any }t\in \pi_1^* (T_S).
\end{equation}
Since $\dim V=r$, there exists an non zero element $v\in V$ such that $\pi_* (v)\in \pi_1^* (T_S)$.
By \eqref{equation2} and \eqref{addremark}, we obtain
\begin{equation}
\frac{ i\Theta_h (L) (v, v )}{\langle v, v\rangle_{\omega_{\tau}}}\geq 
\frac{(c\epsilon_2 -\epsilon )\langle v, v\rangle_{\omega_X}+ \frac{d\cdot\epsilon_1}{2}\langle \pi_* (v), \pi_* (v)\rangle_{\omega_T}}{\tau\langle v, v\rangle_{\omega_X}+\langle \pi_* (v), \pi_* (v)\rangle_{\omega_T}}
\geq \min\{ \frac{c \epsilon_2 -\epsilon}{\tau}, \frac{d\cdot \epsilon_1}{2}\}.
\end{equation}
}
$$\sup\limits_{v\in V}\frac{ i\Theta_h (L) (v, v )}{\langle v, v\rangle_{\omega_{\tau}}}\geq 
 \min \{ \frac{c \epsilon_2 -\epsilon}{\tau}, \frac{d\cdot \epsilon_1}{2}\}\gg (r-1)\cdot \epsilon_2 $$
by the choice of $\tau , \epsilon_1, \epsilon_2$.
Moreover, since $\epsilon_{2}\ll \epsilon_{1}$, 
\eqref{equation2} implies that $i\Theta_{h}(L)$ has at most $(r-1)$-negative eigenvectors
and their eigenvalues with respect to $\omega_{\tau}$ are larger than $ - \epsilon_{2}$.
By the minimax principle, 
$( i\Theta_{h}(L), \omega_{\tau} )$ satisfies the condition \eqref{equation1} in Lemma \ref{keyvanishing1}.
Thus we have
$$H^{q}(X,K_{X}\otimes L\otimes\mathcal{I} (h))=0
\qquad \forall \ q \geq r.$$
Since $\epsilon_{1}, \epsilon_{2}$ are small enough, 
we have $\mathcal{I}(h)=\mathcal{O}_{X}$.
Therefore we get
$$
H^{q}(X,K_{X}\otimes L)=0\qquad \forall \  q\geq r.
$$
\end{proof}

To prove the main theorem in this section, 
we need another vanishing lemma. 
The idea of the proof is essentially the same as Proposition \ref{lemmavanishing}.

\begin{lemma}\label{lemmavanishing3}
Let $(X, \omega_{X})$ be a compact K\"ahler manifold of dimension $n$
which admits a fibration $\pi: X\rightarrow T$ onto a compact K\"ahler manifold $(T, \omega_T)$ of dimension $r$.
Let $L$ be a line bundle on $X$ that is nef and $\pi$-big, and let $A$ be a line bundle on $T$
that is semiample. If $\nd (A)=s$, then we have 
$$
H^q (X, K_X \otimes L \otimes \pi^* (A) )=0 \qquad \forall \ q\geq r-s+1.
$$
\end{lemma}

\begin{proof}
Since $A$ is semiample of numerical dimension $s$,
there exists a smooth metric $h_A$ on $A$ such that $i \Theta_{h_A} (A)$ is semipositive
and has $s$ strictly positive eigenvalues which admit a positive lower bound that does not depend on 
the point $t \in T$.
By the proof of Proposition \ref{lemmavanishing}, 
there exists a metric $h_{2}$ on $L$ with analytic singularities
such that
$$i\Theta_{h_{2}}(L)\geq c\cdot \omega_{X}- \pi^{*} \omega_{T}$$
in the sense of currents for some constant $c > 0$.
Note that $L$ is nef. Then for any $\epsilon> 0$, there exists a smooth metric $h_{\epsilon}$ on $L$
such that $i\Theta_{h_{\epsilon}}(L)\geq -\epsilon\omega_{X}$.
Now we define a new metric $h$ on $L$:
$$h=\epsilon_{2} h_{2}+ (1-\epsilon_{2})h_{\epsilon}$$
for some $\epsilon_{2}\ll 1$ and $\epsilon\ll c\cdot \epsilon_2$.
By construction, we have
\begin{equation}\label{equationadd2}
i\Theta_{h\cdot h_A}(L+\pi^* (A) )=
\epsilon_{2}i\Theta_{h_{2}}(L)+(1-\epsilon_{2})i\Theta_{h_{\epsilon}}(L)+\pi^* ( i\Theta_{h_A}(A) )
\end{equation}
$$\geq -\epsilon_{2}\pi^{*}(\omega_{T})+(c \cdot\epsilon_{2}-\epsilon)\omega_{X} + \pi^* (i\Theta_{h_A}(A) )
= (c \cdot\epsilon_{2}-\epsilon)\omega_{X} + \pi^* (i\Theta_{h_A}(A)-\epsilon_{2}\pi^{*}\omega_{T}) .$$
Since $i\Theta_{h_A}(A)$ is fixed, we can let 
$\epsilon_2$ small enough with respect to the smallest strictly positive eigenvalues of $i\Theta_{h_A}(A)$.
Set $\omega_{\tau}=\tau\omega_X+\pi^* (\omega_T)$ for $\tau> 0$.
Since the semipositive $(1,1)$-form $i\Theta_{h_A}(A)$ contains $s$ strictly positive directions,
by the same argument as in Proposition \ref{lemmavanishing}, we know that the pair
$$( i\Theta_{h\cdot h_A}(L \otimes A) , \omega_{\tau})$$
satisfies the condition \eqref{equation1} in Lemma \ref{keyvanishing1} 
for $p=r-s+1$ when $\tau$ is small enough.
Using Lemma \ref{keyvanishing1}, we obtain that 
$$H^q (X, K_X\otimes L\otimes A\otimes \mathcal{I}(h\cdot h_A))=0 \qquad\text{for }q\geq n-s+1 .$$
Since $\epsilon_2\ll 1$, we have $\mathcal{I}(h\cdot h_A)=\mathcal{O}_X$.
\end{proof}

We can now prove the main theorem in this section:

\begin{theorem}\label{KVvanishing}
Let $X$ be a compact K\"{a}hler manifold of dimension $n$. 
Suppose that there exists a fibration 
$\pi: X\rightarrow T$ onto a torus of dimension $r$.
Let $L$ be a line bundle on $X$ that is nef and $\pi$-big.
If $\nd (L) \geq n-r+1$, then we have
$$
H^{q}(X, K_{X} \otimes L)=0 \qquad \forall \  q \geq r.
$$
\end{theorem}

\begin{remark}
By the same argument as in Proposition \ref{lemmavanishing}, we can easily prove that 
if $\nd (L) =n-r$, then 
$$
H^{q}(X, K_{X} \otimes L)=0 \qquad \forall \ q> r.
$$
\end{remark}

\begin{proof}[Proof of Theorem \ref{KVvanishing}]
Since $\nd (L)\geq n-r+1$, 
Lemma \ref{lemmanumericaldimension} implies that 
\begin{equation} \label{eqna}
\int_{T}\pi_{*}(c_{1}(L)^{n-r+1})\wedge\omega_{T}^{r-1}> 0
\end{equation}
for any K\"ahler class $\omega_{T}$.
Using the assumption that $T$ is a torus, 
we can represent the cohomology class $\pi_{*}(c_{1}(L)^{n-r+1})$ by
a constant $(1,1)$-form 
$\sum_{i=1}^{r}\lambda_{i}d z_{i}\wedge d\overline{z}_{i}$
on $T$.
Since \eqref{eqna} is valid for any K\"ahler class $\omega_{T}$, an elementary computation shows that 
$\lambda_{i}\geq 0$ for any $i$. Thus 
$\pi_{*}(c_{1}(L)^{n-r+1})$ is a semipositive 
(non trivial) class in $H^{1,1} (T) \cap H^2(T, \mathbb{Q})$.
Using \cite[Prop. 2.2]{Cao12}, we get a smooth fibration
$\varphi: T \rightarrow S$ 
where $S$ is an abelian variety of dimension $s$, and
$$\pi_{*}(c_{1}(-K_X)^{n-r+1})=\lambda \  \varphi^* A$$
for some $\lambda>0$ and a very ample divisor $A$ on $S$.

For every $p \in \{ 0, \ldots, s-1 \}$ let $S_{p}$ be a complete intersection of $p$ general divisors in 
the linear system $|A|$, and set $X_p:=\fibre{(\varphi \circ \pi)}{S_p}$ and $T_p:=\fibre{\varphi}{S_p}$.
Then we get a tower of fibrations
$$\begin{CD}
   X_{p} @>\pi|_{X_p}>> T_{p} @>\varphi|_{T_p}>> S_{p}
  \end{CD}
$$
and $X_{p}$ is smooth by Bertini's theorem.
Moreover, we have also the equality
$$
(\pi|_{X_p})_{*}(c_{1}(L)^{n-r+1})=\lambda\cdot (\varphi|_{T_p})^{*} A|_{S_p}. 
$$
Note that $(\varphi|_{T_p})^{*} A|_{S_p}$ is semiample of numerical dimension $\dim S_p$, so
we have
\begin{equation}\label{intervanishing}
H^q (X_p, K_{X_p} \otimes (L \otimes \pi^* \varphi^* A)|_{X_p})=0 \qquad \forall \ q\geq \dim T- \dim S +1.
\end{equation}
by Lemma \ref{lemmavanishing3}.
Using \eqref{intervanishing} and the exact sequence 
$$
0 \rightarrow K_{X_p} \otimes L|_{X_{p}} \rightarrow K_{X_p} \otimes (L \otimes \pi^* \varphi^* A)|_{X_p} \rightarrow K_{X_{p+1}} \otimes L|_{X_{p+1}} \rightarrow 0,
$$
an easy induction shows that
\begin{equation}\label{surjective}
H^{q-(s-1)}(X_{s-1}, K_{X_{s-1}} \otimes L|_{X_{s-1}}) \twoheadrightarrow
H^{q}(X, K_X \otimes L) \qquad \forall \ q \geq r. 
\end{equation}
Applying Proposition \ref{lemmavanishing} to $X_{s -1}$ and the line bundle $L|_{X_{s-1}}$, 
we get
$$
H^{q}(X_{s -1}, K_{X_{s -1}} \otimes L|_{X_{s-1}})=0 \qquad \forall \ q \geq \dim T_{s -1}.
$$
Since $\dim T_{s-1} = r-(s-1)$ we conclude by \eqref{surjective}.
\end{proof}

\begin{proof}[Proof of Proposition \ref{propositionnumericaldimension}]
If $\nd (-K_X)\geq n-r+1$, by taking $L=-K_X$ in Theorem \ref{KVvanishing}, 
we obtain $H^{r}(X, \sO_X)=0$. 
The pull-back of a non-zero holomorphic $r$-form from $T$ gives an contradiction.
\end{proof}

\section{Positivity of direct image sheaves} \label{sectionpositivity}

In this section we will prove that, under the conditions of Theorem \ref{theoremmain}, the direct image sheaves $\pi_* (\omega_X^{\otimes -m})$ are numerically flat vector bundles for $m \gg 0$. If the torus $T$ is projective this can be done by proving
that $\pi_* (\omega_X^{\otimes -m})$ is nef and numerically trivial on a general complete intersection curve. If $T$ is non-algebraic such a curve
does not exist, so we have to refine the construction. In Subsection \ref{subsectionsemistable} we introduce the tools necessary
to deal with the non-algebraic case, Subsection \ref{subsectionanticanonical} contains the core of our proof, the direct image argument.

\subsection{Semistable filtration on compact K\"ahler manifold} \label{subsectionsemistable}

Let us recall the following terminology:

\begin{definition} \cite[Defn.1.5.1]{HL97} \label{definitionjordanhoelder}
Let $(X, \omega_X)$ be a compact K\"ahler manifold, 
and let $\sG$ be a reflexive sheaf that is semistable with respect to $\omega_X$.
A Jordan-H\"older filtration is a filtration
$$
0 = \sG_0 \subset \sG_1 \subset \ldots \subset \sG_l = \sG
$$ 
such that the graded pieces $\sG_{i}/\sG_{i-1}$ are stable for all $i \in \{ 1, \ldots, l\}$.
\end{definition}

Every semistable sheaf admits a Jordan-H\"older filtration \cite[Prop.1.5.2]{HL97}. Moreover given a torsion-free $E$
with Harder-Narasimhan filtration
$$
0 = \sE_0 \subset \sE_1 \subset \ldots \subset \sE_m = E
$$
we can use the Jordan-H\"older filtration of every graded piece $\sE_{j}/\sE_{j-1}$ to obtain a refined filtration
\begin{equation} \label{stablefiltration}
0 = \sF_0 \subset \mathcal{F}_{1}\subset \mathcal{F}_{2} \subset \ldots \subset \mathcal{F}_{k}=E
\end{equation}
such that the graded pieces $\sF_{i}/\sF_{i-1}$ are stable for all $i \in \{ 1, \ldots, k \}$.
We call $\sF_\bullet$ the {\em stable filtration} of $E$ with respect to $\omega$.

\begin{lemma} \label{lemmafiltration}
Let $(X, \omega_X)$ be a compact K\"ahler manifold of dimension $n$, 
and let $\pi: X\rightarrow Y$ be a smooth fibration onto a curve $Y$.
Let $E$ be a nef vector bundle on $X$. Suppose that
$$
c_{1}(E)= M\cdot\pi^{*}\omega_{Y}
$$ 
for some constant $M$ and $\omega_Y$ a K\"ahler form on $Y$.
Let $\sF_\bullet$ be the stable filtration \eqref{stablefiltration} of $E$
with respect to $\pi^{*}\omega_{Y}+\epsilon\omega_{X}$ 
for some $0<\epsilon \ll 1$. 
Then 
$$ 
c_{1}(\mathcal{F}_{1})=a_{1}\cdot\pi^{*}(\omega_{Y}) 
$$
for some $a_{1}\geq 0$.
\end{lemma}

\begin{proof}
If $E$ is stable the statement is obvious, so  suppose that $k \geq 2$.
For $y \in Y$ an arbitrary point, we denote by $X_y$ the fibre over $y$.
Since $c_{1}(E)=  \lambda \cdot\pi^{*}(\omega_{Y})$,
we have $c_{1}(E|_{X_y})=0$.
Then $E|_{X_y}$ is numerically flat, and by the proof of \cite[Thm 1.18]{DPS94},
for any reflexive subsheaf $\mathcal{F}\subset E$, we have
$$c_{1}(\mathcal{F}|_{X_y})\wedge (\omega_{X}|_{X_y})^{n-2}\leq 0 .$$
Therefore we get
\begin{equation} \label{equationfourminus}
c_{1}(\mathcal{F})\wedge \pi^{*} (\omega_{Y}) \wedge\omega_{X}^{n-2}\leq 0 \qquad \forall \ \mathcal{F}\subset E.
\end{equation}
Arguing as \cite[Lemma 1.1]{Cao13}, we see that
\begin{equation} \label{equationfour}
\sup\{ c_{1}(\mathcal{F})\wedge \pi^* (\omega_{Y}) \wedge(\omega_{X})^{n-2} 
\mid \mathcal{F}\subset E\text{ and }c_{1}(\mathcal{F})\wedge \pi^* (\omega_{Y}) \wedge(\omega_{X})^{n-2}< 0\} < 0.
\end{equation}
We claim that 
\begin{equation} \label{equationfive}
c_{1}(\mathcal{F}_{1})\wedge\pi^{*}(\omega_{Y})\wedge \omega_{X}^{n-2}=0 \qquad\text{and}\qquad
c_{1}(\mathcal{F}_{1})\wedge(\omega_{X})^{n-1}\geq 0.
\end{equation}
To prove the claim, we first notice that the nefness of $E$ implies that
\begin{equation} \label{equationstar}
c_{1}(\mathcal{F}_{1})\wedge(\pi^{*}\omega_{Y}+\epsilon\omega_{X})^{n-1}\geq 0. 
\end{equation}  
The base $Y$ being a curve we have
$$(\pi^{*}\omega_{Y}+\epsilon\omega_{X})^{n-1}=
\epsilon^{n-2}\pi^{*}(\omega_{Y})\wedge(\omega_{X})^{n-2}+\epsilon^{n-1}(\omega_{X})^{n-1}.$$
Note also that $c_{1}(\mathcal{F})\wedge(\omega_{X})^{n-1}$ is uniformly bounded from above for 
any $\mathcal{F}\subset E$, cf. \cite[Lemma 7.16]{Kob87}.
Then \eqref{equationstar} implies that 
$$c_{1}(\mathcal{F}_{1})\wedge\pi^{*}(\omega_{Y})\wedge(\omega_{X})^{n-2}\geq -\epsilon\cdot M$$
for a constant $M$ independent of $\epsilon$. 
Since $\epsilon$ is sufficiently small, the uniform estimate $\eqref{equationfour}$ and \eqref{equationfourminus} and imply that 
$$
c_{1}(\mathcal{F}_{1})\wedge\pi^{*}(\omega_{Y})\wedge \omega_{X}^{n-2}=0 .
$$
Using \eqref{equationstar} we deduce that $c_{1}(\mathcal{F}_{1})\wedge(\omega_{X})^{n-1}\geq 0$.
This proves the claim.

Combining \eqref{equationfive} with the assumption that $c_{1}(E)= M\cdot\pi^{*}\omega_{Y}$,
we get
$$c_{1}(E/\mathcal{F}_{1})\wedge\pi^{*}(\omega_{Y})\wedge \omega_{X}^{n-2}=0 .$$
Combining this with the fact that $c_{1}(E/\mathcal{F}_{1})$ is nef and $\omega_{Y}^{2}=0$, 
we obtain
\begin{equation} \label{equationsix}
c_{1}(E/\mathcal{F}_{1})= b \cdot\pi^{*}(\omega_{Y}) 
\end{equation}
for some $b \geq 0$ by Hodge index theorem 
(cf. Remark \ref{corHT} of Appendix \ref{appendixinequality}).
Thus we have $c_{1}(\sF_{1})= a_1 \cdot\pi^{*}(\omega_{Y})$ for some $a_1 \in \Q$ and
\eqref{equationfive} implies that $a_1 \geq 0$. 
\end{proof}

We come to the main result of this subsection:

\begin{proposition} \label{propositionfiltration}
In the situation of Lemma \ref{lemmafiltration}
the reflexive sheaves $\sF_i$ are subbundles of $E$, in particular they and the
graded pieces $\sF_{i}/\sF_{i-1}$ are locally free.
Moreover each of the graded pieces $\sF_{i}/\sF_{i-1}$ is projectively flat, and 
there exist a smooth metric $h_{i}$ on $\sF_{i}/\sF_{i-1}$, such that
$$i\Theta_{h_{i}}(\mathcal{F}_{i+1}/\mathcal{F}_{i})
=a_{i}\pi^{*}(\omega_{Y})\cdot \Id_{\mathcal{F}_{i+1}/\mathcal{F}_{i}}, $$
for some constant $a_{i}\geq 0$.
\end{proposition}

\begin{proof}

{\em Step 1. Proof of the first statement.}
We first prove the statement for $i=1$.
By \cite[Lemma 1.20]{DPS94} it is sufficient to prove that the induced morphism
$$
\det \sF_1 \rightarrow \bigwedge^{\rk \sF_1} E
$$
is injective as a morphism of vector bundles. Note now that the set $Z \subset X$ where
$\sF_1 \subset E$ is not a subbundle has codimension at least two: it is contained 
in the union of the loci where the torsion-free sheaves $\sF_{k+1}/\sF_k$ are not locally free. In particular $Z$ does not
contain any fibre $X_y:=\fibre{\pi}{y}$ with $y \in Y$. Thus for every $y \in Y$ the restricted morphism
\begin{equation} \label{inclusiondet}
(\det \sF_1)|_{X_y} \rightarrow (\bigwedge^{\rk \sF_1} E)_{X_y}
\end{equation}
is not zero. Yet by Lemma \ref{lemmafiltration} the line bundle $(\det \sF_1)|_{X_y}$ is numerically trivial
and the vector bundle $(\bigwedge^{\rk \sF_1} E)_{X_y}$ is numerically flat. Thus the inclusion \eqref{inclusiondet}
is injective as a morphism of vector bundles.
Then $\sF_1$ is a subbundle of $E$ \cite[Prop.1.16]{DPS94}.

Now $E/\sF_1$ is a nef vector bundle on $X$. Moreover, Lemma \ref{lemmafiltration} implies that 
$c_1 (E/\sF_1)=M'\cdot \pi^* (\omega_Y)$ for some constant $M'$.
Then we can argue by induction on $E/\sF_1$, and the first statement is proved.

{\em Step 2. The graded pieces are projectively flat.}
Applying Lemma \ref{lemmafiltration} to $E/\sF_i$, we obtain that $c_1(\sF_{i}/\sF_{i-1})=a_{i}\cdot \pi^* (\omega_Y)$ for some constant $a_i$, in particular
$c_1^2 (\sF_{i}/\sF_{i-1} )=0$.
Since $\sF_{i}/\sF_{i-1}$ is $(\pi^{*}\omega_{Y}+\epsilon\omega_{X})$-stable,
to prove that  $\sF_{i}/\sF_{i-1}$ is projectively flat, by \cite[Thm.4.7]{Kob87} it is sufficient to prove
that 
$$
c_2(\sF_{i}/\sF_{i-1}) \cdot (\pi^{*}\omega_{Y}+\epsilon\omega_{X} )^{n-2}=0 .
$$
Since $c_1(\sF_{i}/\sF_{i-1})$ is a pull-back from the curve $Y$ for every $i \in \{ 1, \ldots, k\}$ it is easy to see that
$$
c_2(E) = \sum_{i=1}^{k} c_2(\sF_{i}/\sF_{i-1}).
$$
Since we have $c_2(\sF_{i}/\sF_{i-1}) \cdot (\pi^{*}\omega_{Y}+\epsilon\omega_{X})^{n-2} \geq 0$ 
for every $i \in \{ 1, \ldots, k\}$ by \cite[Thm.4.7]{Kob87}, 
we are left to show that $c_2(E) \cdot (\pi^{*}\omega_{Y}+\epsilon\omega_{X})^{n-2}=0$. 
Yet $E$ is nef with $c_1(E)^2=0$, 
so this follows immediately from the Chern class inequalities
for nef vector bundles \cite[Cor.2.6]{DPS94}.
\end{proof}

\subsection{Positivity of $\pi_*(\omega_X^{\otimes -m})$}
\label{subsectionanticanonical}

Let $X$ be a normal compact K\"ahler space with at most terminal Gorenstein singularities, 
and let \holom{\pi}{X}{T} be a fibration such that $-K_X$ is $\pi$-nef and $\pi$-big, that
is $-K_X$ is nef on every fibre and big on the general fibre. In this case
the relative base-point free theorem holds \cite[Thm.3.3]{Anc87}, i.e. for every $m \gg 0$ the natural map
$$\pi^{*}\pi_{*}(\omega_X^{\otimes -m})\rightarrow \omega_X^{\otimes -m}$$
is surjective. Thus $\omega_X^{\otimes -m}$ is $\pi$-globally generated and induces a bimeromorphic morphism
\begin{equation} \label{eqnrelativemodel}
\holom{\mu}{X}{X'}
\end{equation}
onto a normal compact K\"ahler space $X'$. Standard arguments from the MMP show that the bimeromorphic
map $\mu$ is crepant, that is $K_{X'}$ is Cartier and we have
$$
K_X \simeq \mu^* K_{X'}.
$$
In particular $X'$ has at most canonical Gorenstein singularities. The fibration $\pi$ factors through the morphism $\mu$,
so we obtain a fibration
\begin{equation} \label{eqnrelativefibration}
\holom{\pi'}{X'}{T}
\end{equation}
such that $-K_{X'}$ is $\pi'$-ample. Therefore we call \holom{\mu}{X}{X'} the relative anticanonical model of $X$ and
\holom{\pi'}{X'}{T} the relative anticanonical fibration.

We start with an elementary computation:

\begin{lemma} \label{lemmaelementary}
Let $V$ be a nef vector bundle over a smooth curve $C$, and let $A \subset V$ be an ample subbundle.
Let $Z \subset \PP(V)$ be a subvariety such that $Z \not\subset \PP(V/A)$. Then we have
$$
Z \cdot \sO_{\PP(V)}(1)^{\dim V}>0.
$$
\end{lemma}

\begin{proof} 
Let $\holom{f}{\PP(V)}{C}$ and $\holom{g}{\PP(A)}{C}$ be the canonical projections, and let 
$\holom{\mu}{X}{\PP(V)}$ be the blow-up along the subvariety $\PP(V/A)$. 
The restriction of $\mu$ to any $f$-fibre \fibre{f}{c} is the blow-up of a projective space $\PP(V_c)$
along the linear subspace $\PP(V_c/A_c)$, so we see that we have a fibration
$\holom{h}{X}{\PP(A)}$ which makes $X$ into a projective bundle over $\PP(A)$.
$$
\xymatrix{
Z'\ar[d]\ar[r]& X \ar[d]^{\mu}\ar[r]^{h} & \PP(A)\ar[ldd]^{g}\\
Z\ar[r]       & \PP(V)\ar[d]^{f} & \\
              & C
}
$$
Since $Z \not\subset \PP(V/A)$, the strict transform $Z'$ is well-defined and we have
$$
\sO_{\PP(V)}(1)^{\dim Z}  \cdot Z = \mu^* \sO_{\PP(V)}(1)^{\dim Z} \cdot Z'.
$$
We claim that
\begin{equation} \label{decompoone}
\mu^* \sO_{\PP(V)}(1) \simeq h^* \sO_{\PP(A)}(1) + E,
\end{equation}
where $E$ is the $\mu$-exceptional divisor. Indeed we can write
$$
\mu^* \sO_{\PP(V)}(1) \simeq a h^* \sO_{\PP(A)}(1) + b E + c F,
$$
where $F$ is a $f \circ \mu$-fibre and $a,b,c \in \Q$. By restricting to $F$ one easily sees 
that we have $a=1, b=1$. Note now (for example by looking at the relative Euler sequence) that we have
$
N_{\PP(V/A)/\PP(V)} \simeq f^* A^* \otimes \sO_{\PP(V/A)}(1)$.
Since the exceptional divisor $E$ is the projectivisation of $N_{\PP(V/A)/\PP(V)}^*$ we deduce that
$$
- E|_E \simeq \sO_{\PP(f^* A \otimes \sO_{\PP(V/A)}(-1))}(1) \simeq (h^* \sO_{\PP(A)}(1))|_E + \mu|_E^* \sO_{\PP(V/A)}(-1).
$$
Since $\mu^* \sO_{\PP(V)}(-1)|_E \simeq \mu|_E^* \sO_{\PP(V/A)}(-1)|_E$ we obtain $c=0$.

In order to simplify the notation, set $\xi_V:=\mu^* \sO_{\PP(V)}(1)$ and $\xi_A:=h^* \sO_{\PP(A)}(1)$.
By \eqref{decompoone} we have
$$
\xi_V^{\dim Z} \cdot Z' = \xi_V^{\dim Z-1} \cdot  \xi_A \cdot Z'
+ \xi_V^{\dim Z-1} \cdot E \cdot Z'.
$$
Since $E$ does not contain $Z'$, the two terms on the right hand side are non-negative. If $\xi_V^{\dim Z-1} \cdot E \cdot Z'>0$
we are obviously finished, so suppose that this is not the case. Let $e$ be the dimension of $h(E \cap Z')$. 
Since $\sO_{\PP(A)}(1)$ is ample and $\xi_V$ is $h$-ample, we have
$$
\xi_V^{\dim Z-e-1} \cdot \xi_A^{e} \cdot E \cdot Z'>0.
$$
Thus
$$
l := \min \{ j \in \N \ | \ \xi_V^{\dim Z-j-1} \cdot \xi_A^j \cdot E \cdot Z'>0 \}.
$$
is an integer. An easy induction now shows that
$$
\xi_V^{\dim Z} \cdot Z' = \xi_V^{\dim Z-l-1} \cdot  \xi_A^{l+1} \cdot Z' +
\xi_V^{\dim Z-l-1} \cdot \xi_A^l \cdot E \cdot Z'>0.
$$
\end{proof}

\begin{lemma} \label{lemmaflat}
Let $X$ be a compact K\"ahler manifold of dimension $n$ such that $-K_X$ is nef.
Let $\pi: X \rightarrow T$ be the Albanese fibration, and suppose that $-K_X$ is $\pi$-big. 
Let $\holom{\pi'}{X'}{T}$ be the relative anticanonical fibration. 

Then $\pi'$ is flat and $E_{m}:=(\pi')_{*}(\omega_{X'}^{\otimes -m})$ is locally free for $m \in \N$.
\end{lemma}

\begin{proof}
The variety $X'$ has at most canonical singularities, so it is Cohen-Macaulay. 
The base $T$ being smooth it is sufficient to prove that $\pi'$ is equidimensional \cite[III,Ex.10.9]{Har77}. 
Set $r:=\dim T$. 
By Proposition \ref{propositionnumericaldimension} we know that 
$$
(-K_X)^{n-r+1}=(-K_{X'})^{n-r+1}=0.
$$
If $F' \subset X'$ is an irreducible component of a $\pi'$-fibre, we have
$(-K_{X'}|_{F'})^{\dim F'} \neq 0$
since $-K_{X'}|_{F'}$ is ample. By the preceding equation we see that $\dim F' \leq n-r$.

Since $X'$ has at most canonical singularities, 
the relative Kawamata-Viehweg theorem applies and shows that
$R^j (\pi')_{*}(\omega_{X'}^{\otimes -m})=0$ for all $j>0$.
The fibration $\pi'$ being flat, the statement follows.
\end{proof}

\begin{lemma} \label{lemmanef}
In the situation of Lemma \ref{lemmaflat}, the vector bundle $E_{m}$ is nef for $m \gg 1$.
\end{lemma}

\begin{remark}
If the fibration is smooth and the torus $T$ is abelian, 
the nefness is proved in \cite[Lemma 3.21]{DPS94}.
Since $T$ is an arbitrary torus and $\pi$ is -a priori- not necessarily smooth, 
we use \cite[Thm. 0.5]{DP04} and the standard regularization method (cf. \cite[Ch.13]{Dem12}, \cite[Sect.3]{Dem92})
to overcome these difficulties.  
\end{remark}

\begin{proof}[Proof of Lemma \ref{lemmanef}]

We first notice that $-K_X =\mu^* (-K_{X'})$ by construction.
Therefore $E_{m}=\pi_{*}(\omega_X^{\otimes -m})$.
We first fix a Stein cover $\mathcal{U}=\{U_{i}\}$ on $T$ as constructed in \cite[13.B]{Dem12}
\footnote{We keep the notations in \cite[13.B]{Dem12}, which can also be found in \cite[Sect. 3]{Dem92} .}, 
such that $U_{i}$ are simply connected balls of radius $2\delta$ fixed.
Let $U_i^{'}\Subset U_i{''}\Subset U_{i}$ be the balls constructed in \cite[13.B]{Dem12} such that they are the balls of radius 
$\delta$, $\frac{3}{2}\delta$, $2\delta$ respectively
and $\{ U_i^{'} \}$ also covers $T$.
Let $\theta_{j}$ be a smooth partition function with support in $U_{j}''$ as constructed in \cite[Lemma 13.11]{Dem12}.
Let $\varphi_{k}: T\rightarrow T$ be a $2^k$-degree isogeny of the torus $T$, and set $X_k :=T\times_{\varphi_k} X$.
Let $L=-(m+1) K_{X_k/T}$ and set $E_{m, k} :=\pi_{*}(K_{X_k}+L)$.
We have the commutative diagram
$$\xymatrix{
&X_k\ar[r]^{\widetilde{\varphi}_k}\ar[d]_{\widehat{\pi}}& X\ar[d]_{\pi}
\\ \PP (E_{m, k})\ar[r]^{\pi_{1}} &T\ar[r]^{\varphi_k} & T 
\\ & U_i\ar@{^{(}->}[u] &}$$
Note that the cover $\mathcal{U}=\{U_{i}\}$, and the partition functions $\theta_{i}$ are independent of $k$.
We first prove that there exists a smooth metric $h$ on $\mathcal{O}_{\PP(E_{m, k})} (1)$, 
such that 
$$i\Theta_{h}(\mathcal{O}_{\PP(E_{m, k})} (1))\geq -C\cdot \pi_1 ^{*}( \omega_{T} )$$
for a constant $C$ independent of $k$
\footnote{ All the constants $C, C_{1},\cdots, C_{i}$ below are independent of $k$.}.

We fix a K\"ahler metric $\omega_{X_k}$ on $X_k$.
Since $L$ is nef and $\widehat{\pi}$-big, 
\cite[Thm. 0.5]{DP04} implies the existence of a singular metric $h_{\widetilde{\epsilon}_k }$ on $L$
such that 
$$i\Theta_{h_{\widetilde{\epsilon}_k}}(L)\geq \widetilde{\epsilon}_k  \omega_{X_k}- C_{1}\widehat{\pi}^*( \omega_{T} ),$$
for a constant $C_{1}$ independent of $k$, but $\widetilde{\epsilon}_k > 0$ is dependent of $k$.
Since $L$ is nef, for any $\epsilon > 0$, there exists a metric $h_{\epsilon}$ such that
$$i\Theta_{h_{\epsilon}}(L)\geq -\epsilon\omega_{X_k} .$$
By combining these two metrics, we can easily construct a new metric $h_{\epsilon_k}$ on $L$,
such that\footnote{We just need to take $h_{\epsilon_k}=h_{\widetilde{\epsilon}_k }^{r_k}\cdot h_{\epsilon}^{1-r_k}$
for some $r_k$ small enough, and $\epsilon\ll r_k\cdot \widetilde{\epsilon}_k$.}
\begin{equation}\label{firstsingularmetric}
i\Theta_{h_{\epsilon_k}}(L)\geq \epsilon_k  \omega_{X_k}- 2\cdot C_{1}\widehat{\pi}^*(\omega_{T})\qquad
\text{and }\qquad \mathcal{I} (h_{\epsilon_k})=\mathcal{O}_{X_k}
\end{equation}
for some $\epsilon_k > 0$.
Since $\mathcal{I} (h_{\epsilon_k})=\mathcal{O}_{X_k}$ and $U_i$ are simply connected Stein varieties, 
we can suppose that $L^2$-bounded (with respect to $h_{\epsilon_k}$) 
elements in $H^{0}(\widehat{\pi}^{-1}(U_{i}), K_{X_k}+L)$
generate $E_{m,k}$ over $U_i$. 

Let $\{\widehat{e}_{i , j}\}_{j}$ be an orthonormal base of $H^{0}(\widehat{\pi}^{-1}(U_{i}), K_{X_k}+L)$ 
with respect to $h_{\epsilon_k}$,
i.e., $\int_{\widehat{\pi}^{-1}(U_{i})} \langle\widehat{e}_{i , j}, \widehat{e}_{i , j'}\rangle_{h_{\epsilon,k}}^2=\delta_{j, j'}$.
Then $\widehat{e}_{i , j}$ induce an element $e_{i,j}\in H^{0}(\pi_{1}^{-1}(U_i), \O_{\PP(E_{m,k})} (1))$.
We now define a smooth metric $h_{i}$ on $\O_{\PP(E_{m,k})} (1) $ over $\pi_{1}^{-1}(U_i)$ by 
$$\|\cdot\|_{h_i}^2=\frac{\|\cdot\|_{h_{0}}^2}{\sum\limits_{j}\|e_{i, j}\|_{h_{0}}^2} ,$$
where $h_{0}$ is a fixed metric on $\O_{\PP(E_m)} (1) $.
Thanks to the construction,
$h_i$ is smooth and $i\Theta_{h_i}(\O_{\PP(E_{m,k})} (1) )$ is semi-positive on $\pi_{1}^{-1}(U_i)$.

We claim that 
\begin{equation}\label{equation75}
\frac{1}{C_{2}}\leq \frac{\sum\limits_{j} \|e_{i, j}\|_{h_{0}}^2 (z)}{\sum\limits_{j} \|e_{i', j}\|_{h_{0}}^2 (z)}\leq C_{2}
\qquad \text{for }z\in \pi_{1}^{-1}(U_{i}''\cap U_{i'}'') ,
\end{equation}
for some $C_{2}> 0$ independent of $z, k, i, i'$.
The proof is almost the same as in \cite[Lemma 13.10]{Dem12}, except that we use the metric 
$\epsilon_k\cdot\omega_{X_k}+\widehat{\pi}^{*}\omega_T$
in stead of $\omega_X$ in the estimate.
We postpone the proof of the claim \eqref{equation75} in Lemma \ref{gluing}
and first finish the proof of Lemma \ref{lemmanef}.

We now define a global metric $h$ on $\mathcal{O}_{\PP(E_{m})} (1)$ by
$$\|\cdot\|_{h}^{2}=\|\cdot\|_{h_{0}}^{2} e^{- \sum\limits_{i} (\pi_1 ^{*}(\theta_{i}'))^{2}\cdot\ln (\sum\limits_{j} \|e_{i, j}\|_{h_{0}}^2)} ,
\qquad\text{where }(\theta_{i}')^{2}=\frac{\theta_{i}^{2}}{\sum\limits_{k} \theta_{k}^{2}} .$$
Note that
$$i (\theta_{j}'\partial\overline{\partial} \theta_{j}'-\partial\theta_{j}'\wedge \overline{\partial} \theta_{j}' )
\geq -C_{3}\cdot\omega_{T}$$
by construction.
Combining this with \eqref{equation75} and  
applying the Legendre identity in the proof of 
\cite[Lemma 13.11]{Dem12}\footnote{Although in the proof of \cite[Lemma 13.11]{Dem12}, $\theta_{i}'$ is supposed to be constant on $U_i^{'}$,
the uniformly strictly positive of the lower boundedness of $\theta_{i}'$ on $U_{i}'$ is sufficient for the proof.},
we obtain that
$$i\Theta_{h}(\mathcal{O}_{\PP(E_{m,k})} (1))\geq -C\cdot \pi_1 ^{*}(\omega_{T})$$
for a constant $C$ independent of $k$.

By \cite[Prop. 1.8]{DPS94}, the metric $h$ on $\mathcal{O}_{\PP(E_{m,k})} (1)$ induce a smooth metric $h_k$
on $\mathcal{O}_{\PP(E_{m})} (1)$ such that
$$i\Theta_{h_k}(\mathcal{O}_{\PP(E_{m})} (1))\geq -\frac{C}{2^{k-1}}\omega_{T}.$$
The lemma is proved by letting $k\rightarrow +\infty$.
\end{proof}

We now prove the claim \eqref{equation75} in Lemma \ref{lemmanef}, 
which is in some sense a relative gluing estimate.

\begin{lemma}\label{gluing} In the situation of the proof of  Lemma \ref{lemmanef},
we have
\begin{equation}\label{equationlemmagluing}
\frac{1}{C_{2}}\leq \frac{\sum\limits_{j} \|e_{i, j}\|_{h_{0}}^2 (z)}{\sum\limits_{j} \|e_{i', j}\|_{h_{0}}^2 (z)}\leq C_{2}
\qquad \text{for }z\in \pi_{1}^{-1}(U_{i}''\cap U_{i'}'').
\end{equation} 
\end{lemma}

\begin{proof}

Recall that $U_i^{'}\Subset U_i{''}\Subset U_{i}$ are the balls of radius $\delta$, $\frac{3}{2}\delta$, $2\delta$ respectively
as constructed in \cite[13.B]{Dem12}.
Let $z$ be a fixed point in $\pi_{1}^{-1}(U_{i}''\cap U_{i'}'')$.
Since $e_{i,j }$ is a section of a line bundle, we have
$$\sum\limits_{j} \|e_{i, j}\|_{h_{0}}^2 (z)= 
\sup\limits_{\sum\limits_{j} \mid a_j \mid^2 =1} \|\sum\limits_{j} a_{j} e_{i, j}\|_{h_{0}}^{2}(z) .$$
Therefore, there exists a $\widehat{e}_{i}\in H^{0}(\widehat{\pi}^{-1}(U_{i}), K_{X_k}+L)$
such that
\begin{equation}\label{estimateonepiece}
\int_{\widehat{\pi}^{-1}(U_{i})} \|\widehat{e}_{i}\|_{h_{\epsilon_k}}^{2}=1\qquad \text{and}\qquad
\|e_{i}\|_{h_{0}}^{2} (z) = \sum\limits_{j} \|e_{i, j}\|_{h_{0}}^2 (z) ,
\end{equation}
where $e_{i}\in H^{0}(\pi_{1}^{-1}(U_i), \O_{\PP(E_{m,k})} (1))$ is induced by $\widehat{e}_{i}$.
Our goal is to construct an element in $H^{0}(\widehat{\pi}^{-1}(U_{i'}), K_{X_k}+L)$ with controlled norm
and equals to $\widehat{e}_{i}$ on $\widehat{\pi}^{-1}(\pi_{1} (z))$.

Let $\theta$ be a cut-off function with support in the ball of radius $ \frac{\delta}{4}$ 
centered at $\pi_{1} ( z )$ (thus is supported in $U_{i}\cap U_{i'}$), 
and equal to $1$ on the ball of radius $\frac{\delta}{8}$ centered at $\pi_{1} ( z )$.
By construction, $(\widehat{\pi}^{*}(\theta)\cdot\widehat{e}_{i})$ is supported in $\widehat{\pi}^{-1} (U_{i}\cap U_{i'})$, 
thus it is well defined on $\widehat{\pi}^{-1}(U_{i'})$.
Therefore we can solve the $\overline{\partial}$-equation for $\overline{\partial} (\widehat{\pi}^{*}(\theta)\cdot\widehat{e}_{i})$ on 
$\widehat{\pi}^{-1}(U_{i'})$
with respect to the metric 
\begin{equation}\label{newmetric}
\omega_{X_k,\epsilon_k}=\epsilon_k\cdot\omega_{X_k}+\widehat{\pi}^{*}\omega_T 
\end{equation}
by choosing a good metric on $L$. 
Before giving the good metric on $L$, we first give some estimates.

Since $\theta$ is defined on $T$, we have
$$\|\overline{\partial}\widehat{\pi}^{*}(\theta)\|_{\omega_{X_k,\epsilon_k}}\leq C_4$$
for some constant $C_{4}$ independent of $k, \epsilon_k$
\footnote{$C_4$ depends on $\delta$. But by construction, the radius $\delta$ is independent of $k$.}.
Therefore we have
\begin{equation}\label{equation5}
\int_{\widehat{\pi}^{-1}(U_{i'})}\|\overline{\partial}(\widehat{\pi}^{*}(\theta)\cdot\widehat{e}_{i})\|_{h_{\epsilon_k}, \omega_{X_k,\epsilon_k}}^{2}
 =\int_{\widehat{\pi}^{-1}(U_{i})}\|\overline{\partial}(\widehat{\pi}^{*}(\theta)\cdot\widehat{e}_{i})\|_{h_{\epsilon_k}, \omega_{X_k,\epsilon_k}}^{2}
\end{equation}
$$
\leq C_4 \int_{\widehat{\pi}^{-1}(U_{i})}\|\widehat{e}_{i}\|_{h_{\epsilon_k}}^{2}=C_{4} ,
$$
where the first equality comes from the fact that  $(\widehat{\pi}^{*}(\theta)\cdot\widehat{e}_{i})$ is supported in 
$\widehat{\pi}^{-1} (U_{i}\cap U_{i'})$.
By \eqref{firstsingularmetric} and \eqref{newmetric}, we have
\begin{equation}\label{equation6}
i\Theta_{h_{\epsilon_k}}(L)\geq \epsilon_k\omega_{X_k}-2\cdot C_1 \widehat{\pi}^{*}(\omega_{T} )\geq 
\omega_{X_k,\epsilon_k}-(2\cdot C_1 +1) \widehat{\pi}^{*}(\omega_{T} ). 
\end{equation}

We now define a metric $\widetilde{h}_{\epsilon_k} = h_{\epsilon_k}\cdot e^{-(n+1) \widehat{\pi}^{*}(\ln |t-\pi_{1} (z)|)-\widehat{\pi}^{*}\psi_{i'}(t)}$ 
on $ L $ over $\widehat{\pi}^{-1}(U_{i'})$,
where $\psi_{i'}(t)$ is a uniformly bounded function on $U_{i'}$ satisfying 
$$dd^c \psi_{i'}(t)\geq (2 C_1 +1) \omega_{T} .$$
Then \eqref{equation6} implies that 
$$i\Theta_{\widetilde{h}_{\epsilon_k}}(L)\geq \omega_{X_k,\epsilon_k}  \qquad \text{on }\widehat{\pi}^{-1}(U_{i'}).$$
By solving the $\overline{\partial}$-equation for $\overline{\partial} (\widehat{\pi}^{*}(\theta)\cdot\widehat{e}_{i})$ 
with respect to $(\widetilde{h}_{\epsilon_k} , \omega_{X_k,\epsilon_k} )$ on $\widehat{\pi}^{-1}(U_{i'})$, 
we can find a $g_{i'}\in L^{2}(\widehat{\pi}^{-1}(U_{i'}), K_{X_k}+L)$ such that
\begin{equation}\label{solutionpartial}
\overline{\partial}g_{i'}=\overline{\partial}(\widehat{\pi}^{*}(\theta)\widehat{e}_{i})
\end{equation}
and
\begin{equation}\label{equation7}
\int_{\widehat{\pi}^{-1}(U_{i'})}\|g_{i'}\|_{\widetilde{h}_{\epsilon_k}}^{2}\leq 
C_5\int_{\widehat{\pi}^{-1}(U_{i'})} \|\overline{\partial}(\widehat{\pi}^{*}(\theta)
\cdot\widehat{e}_{i})\|_{\widetilde{h}_{\epsilon_k}, \omega_{X,\epsilon_k}}^{2}
\leq C_{6},
\end{equation}
where the second inequality comes from \eqref{equation5} 
and the fact that
$$
\overline{\partial}(\widehat{\pi}^{*}(\theta)\cdot\widehat{e}_{i}) (z)=0 \qquad \text{for }z\in \widehat{\pi}^{-1}(B_{\frac{\delta}{8}}(\pi_{1}(z))),
$$
where $B_{\frac{\delta}{8}}(\pi_{1}(z))$ is the ball of radius $\frac{\delta}{8}$ centered at $\pi_{1}(z)$.
By \eqref{solutionpartial}, we obtain a holomorphic section 
$$\widehat{e}_{i'} := (\widehat{\pi}^{*}(\theta)\cdot \widehat{e}_{i}- g_{i'} )\in H^{0}(\widehat{\pi}^{-1}(U_{i'}), K_{X_k}+L) .$$
By the definition of the metric $\widetilde{h}_{\epsilon_k}$ and \eqref{equation7}, 
$g_{i'}=0$ on $\widehat{\pi}^{-1}(\pi_{1}(z))$.
Therefore 
$\widehat{e}_{i'} = \widehat{e}_{i} $ on $\widehat{\pi}^{-1}(\pi_{1} (z))$.
Moreover, \eqref{estimateonepiece} and \eqref{equation7} imply that
$$\int_{\widehat{\pi}^{-1}(U_{i'})} \|\widehat{e}_{i'}\|_{h_{\epsilon_k}}^{2}=\int_{\widehat{\pi}^{-1}(U_{i'})}\|\widehat{\pi}^{*}(\theta)\cdot\widehat{e}_{i}- g_{i'}\|_{h_{\epsilon_k}}^{2}\leq C$$
for a constant $C$ independent of $k$.
By the extremal property of Bergman kernel, \eqref{equationlemmagluing} is proved.
\end{proof}

\begin{lemma} \label{lemmanumericallyflat}
In the situation of Lemma \ref{lemmaflat}, the vector bundle $E_{m}$ is numerically flat.
\end{lemma}

\begin{proof}
By Lemma \ref{lemmanef} it is sufficient to prove that $c_{1}(E_{m})=0$.
Arguing by contradiction we suppose that $c_{1}(E_{m})\neq 0$.
Then \cite[Prop.2.2]{Cao12} implies that there exists a smooth fibration 
$\pi_{1}: T\rightarrow S$ onto an abelian variety $S$ of dimension $s$ such that 
$$
c_{1}(E_{m})=c\cdot\pi^{*}_{1} c_1(A)
$$
for some very ample line bundle $A$ and $c> 0$.

Let $S_{1}$ be a complete intersection of $s-1$ hypersurfaces defined by $s-1$ general elements in $H^{0}(S, A)$.
We have thus a morphism
$$\begin{CD}
   X_{1} @>\pi|_{X_1}>> T_{1} @>{\pi_1}|_{T_1}>> S_{1}
  \end{CD}
$$
where $X_{1} :=\pi^{-1}\pi_{1}^{-1}(S_{1})$ and $T_{1} :=\pi_{1}^{-1}(S_{1})$
are smooth by Bertini's theorem.
For simplicity of notation we set $E_{m}':=E_{m}|_{T_{1}}$. 
Then $E_{m}'$ is nef and $c_{1}(E_{m}')=c\cdot (\pi_{1}|_{T_1})^* c_1(A)$. 
Applying Proposition \ref{propositionfiltration}, 
we obtain a semipositive projective flat vector bundle $0\subset F_{1}\subset E_{m}'$ on $T_{1}$
such that 
and $c_{1}(F_{1})=\pi_{1}^{*} (\omega_{S_{1}})$
for some K\"ahler form $\omega_{S_{1}}$ on $S_{1}$.

Our goal is to show that 
$$
(\sO_{\PP (E_{m}')}(1))^{n-r+1} \cdot X_{1} > 0.
$$
Since $\sO_{\PP (E_{m}')}(1)|_{X_1} \equiv (-m K_X)|_{X_1}$ this gives a contradiction 
to Proposition \ref{propositionnumericaldimension}. 
Note first that $X_1$ is not contained in any projective subbundle of $\PP (E_{m}')$ since
the general $\pi$-fibre $F$ is embedded by the complete linear system $|-m K_F|$, so it is linearly non-degenerate.
Thus if $\pi_1$ is an isomorphism (which is equivalent to $\det E_m$ being ample) 
the statement follows from Lemma \ref{lemmaelementary}.
In the general case we follow a similar construction: let 
$\mu: Y\rightarrow \mathbb{P}(E_{m}')$ be the blow-up along the subvariety 
$\mathbb{P}(E_{m}'/F_{1})$.
Since $X_{1}$ is not contained in $\mathbb{P}(E_{m}'/F_{1})$, we have a diagram
$$
\xymatrix{
X_{1}'\ar[d]\ar[r]^{i}& Y \ar[d]^{\mu}\ar[r]^{h} & \PP(F_{1})\ar[ldd]^{g}\\
X_{1}\ar[r]    \ar[rd]_{\pi|_{X_1}}   & \PP(E_{m}')\ar[d]^{f} & \\
              & T_{1}\ar[d]^{\pi_{1}}\\
              & S_{1}
}
$$
where $X_{1}'$ is the strict transformation of $X_{1}$ and $f,g$ and $h$ are the natural maps as in the proof 
of Lemma \ref{lemmaelementary}.
By the same argument as in \eqref{decompoone} of Lemma \ref{lemmaelementary}, we have
$$
\mu^* \sO_{\PP (E_{m}')}(1) \simeq h^* \sO_{\PP(F_{1})}(1) + E,
$$
where $E$ is the $\mu$-exceptional divisor.
Since $\sO_{\PP (E_{m}')}(1)$ is nef,
we have
\begin{equation}\label{equationaddd}
(\mu^* \sO_{\PP (E_{m}')}(1))^{n-r+1} \cdot X_{1}'\geq
(\mu^* \sO_{\PP (E_{m}')}(1))^{n-r}\cdot h^* \sO_{\PP(F_{1})}(1)\cdot X_{1}' .
\end{equation}
By Proposition \ref{propositionfiltration}, there is a smooth metric $h_0$ on $F_{1}$
such that 
$$i \Theta_{h_0} (F_{1}) =\pi_{1}^*\omega_{S_1}\cdot \Id_{F_{1}} .$$
On $\PP(F_{1})$ we have
$$
i \Theta_{h_0} (g^*F_{1}) = g^*\pi_{1}^*\omega_{S1}\cdot \Id_{g^*F_{1}}.
$$
The metric $h_0$ induces a natural metric $h_0 ^{'}$ on $\sO_{\PP(F_{1})}(1)$,
and by \cite[Prop. 1.11]{DPS94},
we obtain
$$
i \Theta_{h_0 ^{'}}(\sO_{\PP(F_{1})}(1))\geq g^*\pi_{1}^*\omega_{S_1}.
$$ 
Since $h\circ g=\mu\circ f$, we get
$h^*i \Theta_{h_0 ^{'}} (\sO_{\PP(F_{1})}(1))\geq \mu^{*} f^{*}\omega_{S_{1}}$.
Combining this with the fact that $f\circ\mu\circ i (X_{1}')=T_1$ by construction, 
we obtain 
$$h^* \sO_{\PP(F_{1})}(1)\cdot X_{1}'\geq C\cdot X_{1, s}',$$
where $X_{1,s}'$ is the general fiber of $i\circ\mu\circ f\circ \pi_{1}$, and $C> 0$.
Combining this with the fact that $\sO_{\PP (E_{m}')}(1)$ is $f$-ample,
we get
$$(\mu^* \sO_{\PP (E_{m}')}(1))^{n-r}\cdot h^* \sO_{\PP(F_{1})}(1)\cdot X_{1}' \neq 0 .$$
We conclude by \eqref{equationaddd}.
\end{proof}

\subsection{The projective case}

In Section \ref{sectiondimensiontwo} we will need the following
versions of Lemma \ref{lemmanef} and Lemma \ref{lemmanumericallyflat}.

\begin{lemma} \label{lemmakltdirectimage}
Let $X$ be a normal projective variety and $\Delta$ a boundary divisor on $X$ such that
the pair $(X, \Delta)$ is klt. Let \holom{\varphi}{X}{T} be a flat fibration onto an abelian
variety. Let $L$ be a Cartier divisor on $X$ such that $L-(K_X+\Delta)$ 
is nef and $\varphi$-big. Then
the following holds:
\begin{enumerate}[(i)]
\item Let $H$ be an ample line bundle on $T$, then $\varphi_* (\sO_X(L)) \otimes H$ is ample.
\item The direct image sheaf $\varphi_* (\sO_X(L))$ is nef.
\end{enumerate} 
\end{lemma}

\begin{proof}
Note first that since $L-(K_X+\Delta)$ is relatively nef and big, 
the higher direct images $R^j \varphi_* (\sO_X(L))$
vanish for all $j>0$ by the relative Kawamata-Viehweg theorem. By cohomology and base change the sheaf
$\varphi_* (\sO_X(L))$ is locally free.

{\em Proof of  (i).} Since $H$ is ample and $L-(K_X+\Delta)$ is nef and $\varphi$-big, 
the divisor $L+\varphi^*H-(K_X+\Delta)$ is nef and big.
So if $P \in \mbox{Pic}^0(T)$ is a numerically trivial line bundle, then
by Kawamata-Viehweg vanishing one has
\[
H^j(T, \varphi_* (\sO_X(L)) \otimes H \otimes P) 
=
H^j(X, L \otimes \varphi^*(H \otimes P)) = 0 
\qquad
\forall \ j > 0. 
\]
Thus $ \varphi_* (\sO_X(L)) \otimes H$ satisfies Mukai's property $IT_0$ \cite{Muk81}. 
In particular it is ample by \cite[Cor.3.2]{Deb06}.

{\em Proof of  (ii).} Let \holom{\mu}{T}{T} be the multiplication map $x \mapsto 2x$.
Let $H$ be a symmetric ample divisor, then by the theorem of the square
\cite[Cor.2.3.6]{BL04} we have $\mu^* H \simeq H^{\otimes 4}$. 
It is sufficient to show that for $d \in \N$ arbitrary, the $\Q$-twisted \cite[Sect.6.2]{Laz04b} vector bundle
$\varphi_* (\sO_X(L))\hspace{-0.8ex}<\hspace{-0.8ex}\frac{1}{4^d}H\hspace{-0.8ex}>$
is nef. Denote by $\mu_d$ the $d$-th iteration of $\mu$ and consider the base change diagram
\[
\xymatrix{
X_d
\ar[d]_{\tilde \varphi} \ar[r]^{\tilde \mu_d} & X \ar[d]_\varphi
\\
T \ar[r]^{\mu_d} & T
}
\]
where $X_d := X \times_T T$.
By flat base change \cite[Prop.9.3]{Har77} we have
\[
\mu_d^*(\varphi_* (\sO_X(L))\hspace{-0.8ex}<\hspace{-0.8ex}\frac{1}{4^d}H\hspace{-0.8ex}>)
\sim_\Q
\tilde \varphi_* (\sO_{X_d}(\tilde \mu_d^*L))\hspace{-0.8ex}<\hspace{-0.8ex}\frac{1}{4^d}\mu_d^* H\hspace{-0.8ex}>
\]
Since $\frac{1}{4^d} \mu_d^* H \sim_\Q H$, we see that
$\mu_d^*(\varphi_* (\sO_X(L))\hspace{-0.8ex}<\hspace{-0.8ex}\frac{1}{4^d}H\hspace{-0.8ex}>)
\sim_\Q
\tilde \varphi_* (\sO_{X_d}(\tilde \mu_d^*L)) \otimes H
$
which by the first statement is ample. 
\end{proof}

\begin{proposition} \label{propositionkltdirectimage}
Let $X$ be a normal projective variety and $\Delta$ a boundary divisor on $X$ such that
the pair $(X, \Delta)$ is klt. Let \holom{\varphi}{X}{T} be a flat fibration onto a
smooth curve or an abelian
variety. Suppose that $-(K_{X/T}+\Delta)$ is nef and $\varphi$-ample.
Then 
$$
E_{m}:=\varphi_{*}(\sO_X(-m (K_{X/T}+\Delta)))
$$ 
is a numerically flat vector bundle
for all sufficiently large and divisible $m \gg 0$.
\end{proposition}

\begin{proof} 
For sufficiently divisible $m \in \N$ the $\Q$-divisor $m (K_{X/T}+\Delta)$ is integral
and Cartier.

{\em 1st case. $T$ is a curve.} Then $E_m$ is nef \cite{Kol86},
and for $m \gg 0$ we have an inclusion
$X \hookrightarrow \PP(E_m)$.
We can now argue as in Lemma \ref{lemmanumericallyflat}: if $E_m$ is not numerically flat, there
exists an ample subbundle $A \subset E_m$. By 
Lemma \ref{lemmanumericaldimensionprojective} we have
$(K_{X/T}+\Delta)^{\dim X}=0$, so Lemma \ref{lemmaelementary} implies that $X \subset \PP(E_m/A)$.
However this is not possible since the embedding of the general fibre $X_t$ in $\PP((E_m)_t)$ is linearly nondegenerate.

{\em 2nd case. $T$ is an abelian variety.} By Lemma \ref{lemmakltdirectimage}
the sheaf $E_m$
is a nef vector bundle.
If $C$ is a general complete intersection curve in $T$,
denote by $X_C:=\fibre{\varphi}{C}$ its preimage. Then the pair
$(X_C, \Delta_C:=\Delta \cap X_C)$ is klt and the relative canonical divisor
$-(K_{X_C/C}+\Delta_C)$
is nef and relatively ample. By the first case $\varphi_{*}(E_m \otimes \sO_C)$
is numerically flat, so $\det E_m$ is numerically 
trivial \cite[3.8]{Deb01}.
\end{proof}

\section{Proof of Theorem \ref{theoremmain}}

Let $X'$ be a normal compact K\"ahler space that admits a flat fibration $\pi': X' \rightarrow T$ 
onto a compact K\"ahler manifold $T$. Let $L$ be a line bundle on $X$ that is $\pi'$-ample,
for all $m \in \N$ we set $E_m := (\pi')_* (\sO_{X'}(mL))$. We fix a $m \gg 0$ such that we have
an embedding $X' \hookrightarrow \mathbb{P}(E_{m})$, for simplicity's sake we denote by
$\holom{\pi'}{\PP(E_m)}{T}$ the natural map. Let $\sI_{X'} \subset \sO_{\PP(E_m)}$ be the ideal
sheaf of $X'$. Then we define for every $d \in \N$
$$
S_{m, d} := (\pi')_{*}(\sI_{X'}\otimes \mathcal{O}_{\mathbb{P}(E_{m})}(d)).
$$

\begin{proposition} \label{propositionloctriv}
In the situation above suppose that
$E_{m}$ and $S_{m, d}$ are numerically flat vector bundles
for all $m \gg 1$ and $d\gg 1$. Then the fibration $\pi'$ is locally trivial.
\end{proposition}

\begin{proof}
We have a natural inclusion $i: S_{m, d}\hookrightarrow S^d E_m$.
Since numerically flat vector bundles are local systems 
(cf. \cite[Sect.3]{Sim92} for general case or \cite[Lemma 6.5]{Ver04} for numerically flat vector bundles),
$S_{m, d}$ and $S^d E_m$ are local systems on $T$.
Let $U$ be any small Stein open set in $T$, and let $e_{1},\cdots, e_{k}$ be a local constant coordinates of $S_{m, d}$
over $U$.
Note that $\Hom (S_{m, d}, S^d E_m)$ is also a local system on $T$, and 
$i\in H^{0}(T, \Hom (S_{m, d}, S^d E_m))$.
By \cite[Lemma 4.1]{Cao12}, $i$ is 
parallel\footnote{It is not true that for any local system, the global sections are parallel.
However, it is true for numerically flat bundles.}
with respect to the local system 
$\Hom (S_{m, d}, S^d E_m)$.
Therefore the images of $e_{1},\cdots, e_{k}$ in $S^d E_m$ are also locally constant,
i.e. the polynomials defining the fibers $X'_{t}$ for $t\in U$ are locally constant.
In particular the fibration $\pi'$ is locally trivial.
\end{proof}

\begin{proof}[Proof of Theorem \ref{theoremmain}]
{\em Step 1. The relative anticanonical fibration is locally trivial.}
We follow the argument of \cite[Thm.3.20]{DPS94}. Denote by $X'$ the relative anticanonical model of $X$ 
(cf. Section \ref{subsectionanticanonical}).
Fix $m \gg 0$ such that we have an inclusion $X'\hookrightarrow \mathbb{P}(E_{m})$
where $E_m$ is defined as in Lemma \ref{lemmaflat} and denote by $\pi'$ both the anticanonical fibration
$X' \rightarrow T$ and $\PP(E_m) \rightarrow T$. 
For every $d \in 
\N$ we have an exact sequence
$$
0 \rightarrow \sI_{X'}\otimes \mathcal{O}_{\mathbb{P}(E_{m})}(d) \rightarrow
\mathcal{O}_{\mathbb{P}(E_{m})}(d) \rightarrow  \omega_{X'}^{\otimes -md} \rightarrow 0.
$$
Since $\sO_{\PP(E_{m})}(1)$ is $\pi'$-ample we get for $d \gg 0$ an exact sequence
$$
0 \rightarrow (\pi')_{*}(\sI_{X'}\otimes \mathcal{O}_{\mathbb{P}(E_{m})}(d)) 
\rightarrow S^d E_m \rightarrow E_{md} \rightarrow 0.
$$
The vector bundles $S^d E_m$ and $E_{md}$ are numerically flat, so $(\pi')_{*}(\sI_{X'}\otimes \mathcal{O}_{\mathbb{P}(E_{m})}(d))$
is numerically flat. Now apply Proposition \ref{propositionloctriv}.

{\em Step 2. The Albanese map $\pi$ is locally trivial.} Let $F'$ be the general fibre of the relative anticanonical fibration, and
let $F_t$ be any smooth $\pi$-fibre.
Then $F_t$ is a weak Fano manifold, and $\holom{\mu|_{F_t}}{F_t}{F'}$ is a crepant birational morphism,
so $F_t$ is a terminal model of the Fano variety $F'$.
By \cite[Cor.1.1.5]{BCHM10} there are only finitely many terminal models of $F'$. 
Thus there are only finitely many possible isomorphism classes for $F_t$,
hence there exists a non-empty Zariski open subset $T_0 \subset T$ such that $\fibre{\pi}{T_0} \rightarrow T_0$ 
is locally trivial with fibre $F$. Fix now an ideal sheaf $\sI$ on $F'$ such that $F$ is isomorphic to the blow-up of $F'$
along $\sI$.

Let $t \in T$ be an arbitrary point, and let $t \in U \subset T$ be an analytic neighbourhood such that
$X'_U:=\fibre{(\pi')}{U} \simeq U \times F'$. Let \holom{\tilde \mu}{\tilde X_U}{X'_U} be the blow-up
of $X'_U$ along the ideal sheaf $\sI \otimes \sO_U$, then $\tilde X_U \simeq U \times F$.
In particular $\tilde X_U$ is smooth and the birational morphism $\tilde \mu$ is crepant.
Set $X_U :=\fibre{\pi}{U}$, then $X_U$ is also smooth and $\holom{\mu|_{X_U}}{X_U}{X'_U}$
is crepant. Thus $X_U$ and and $\tilde X_U$ are both minimal models over the base $X'_U$ (cf. \cite[Defn.3.48]{KM98}),
hence the induced birational morphism $\tau: X_U \dashrightarrow X'_U$ is an isomorphism in codimension one \cite[Thm.3.52]{KM98}.
Moreover by the universal property of the blow-up the restriction of $\tau$
to 
$$
(X_U \cap \fibre{\pi}{T_0}) \dashrightarrow (X'_U \cap \fibre{(\pi' \circ \mu)}{T_0})
$$ 
is an isomorphism.
Let $H$ be an effective $\pi$-ample divisor on $X_U$, and let $H':=\tau_* H$ be its strict transform.
Then $H'$ is $\pi' \circ \mu$-nef: indeed if $C$ is any (compact) curve in $\tilde X_U \simeq U \times F$, it
deforms to a curve $C'$ that is contained in $e \times F$ with $e \in (U \cap T_0)$. Yet the restriction 
of $H'$ to $e \times F$ is ample, since $\tau$ is an isomorphism in a neighbourhood of $e \times F$.
Thus we see that $H' \cdot C = H' \cdot C'>0$. Since $X_U$ is smooth we satisfy the conditions
of \cite[Lemma 6.39]{KM98}, so $\tau$ is an isomorphism.
This shows that $X_U \rightarrow U$ is locally trivial with fibre $F$, since $t \in T$ is arbitrary this concludes the proof.
\end{proof}

\begin{proof}[Proof of Corollary \ref{corollarymain}]
Suppose first that $q(X)=\dim X-1$. If $-K_X$ is $\pi$-big we see by Theorem \ref{theoremmain} that the Albanese map is a $\PP^1$-bundle. If $-K_X$ is not $\pi$-big, the general fibre is an elliptic curve. Note that the $C_{n, n-1}$-conjecture is also known
in the K\"ahler case \cite[Thm.2.2]{Uen87}, so we see that $\kappa(X) \geq 0$. Yet $-K_X$ is nef, so we see that $-K_X \equiv 0$.
We conclude by the Beauville-Bogomolov decomposition. 
\end{proof}

The proof of Theorem \ref{theoremmain} works also in the following relative situation which we will use
in Section \ref{sectiondimensiontwo}.

\begin{corollary} \label{corollaryMFS}
Let $X$ be a normal $\Q$-factorial projective variety with at most terminal singularities,
and let $\holom{\varphi}{X}{C}$ be a fibration onto a smooth curve 
such that $-K_{X/C}$ is nef and $\varphi$-big. If the general fibre is smooth,
then $\varphi$ is locally trivial in the analytic topology.
\end{corollary}

\section{Two-dimensional fibres} \label{sectiondimensiontwo}

In this section we will prove Theorem \ref{theoremmaintwo}. While the positivity of direct image sheaves
will still play an important role, we need some additional geometric information to deduce the smoothness of the Albanese map.
This information will be obtained by describing very precisely the MMP.

\subsection{Reduction to the curve case}

The following lemma shows that the nefness condition imposes strong restrictions on the singularities
of a fibre space.

\begin{lemma} \label{lemmaimportant}
Let $M$ be a normal projective variety, and let $\holom{\varphi}{M}{C}$ be a fibration onto a curve such
that $-K_{M/C}$ is $\Q$-Cartier and nef.
Let $\Delta$ be a boundary divisor on $M$ such that $\Delta \equiv -\alpha K_{M/C}$ for some $\alpha \in [0, 1]$.

If the pair $(M, \Delta)$ is lc (resp. klt) over the generic point of $C$, the pair 
$(M, \Delta)$ is lc (resp. klt).
If the pair $(M, \Delta)$ is lc, every lc center of $(M, \Delta)$ surjects onto $C$.
\end{lemma}

\begin{proof}
Let $\holom{\mu}{M'}{M}$ be a dlt blow-up \cite[Thm.10.4]{Fuj11}, i.e. $\mu$ is birational morphism from $M'$ a normal $\Q$-factorial
variety such that if we set
$$
\Delta' := \mu_*^{-1} \Delta + \sum_{E_i \ \mbox{\tiny $\mu$-exc.}} E_i,
$$
then $(M', \Delta')$ is dlt. Moreover one has
$$
K_{M'} +\Delta' = \mu^* (K_M+\Delta) + \sum_{E_i \ \mbox{\tiny $\mu$-exc.}} (a_i+1) E_i,
$$
and $a_i \leq -1$ for all $i$. Let $\Delta'=\Delta'_{hor}+\Delta'_{vert}$ be the decomposition
in the horizontal and vertical part.
Since $M'$ is $\Q$-factorial, the pair $(M', \Delta'_{hor})$ is dlt \cite[Cor.2.39]{KM98} and
$$
K_{M'} +\Delta'_{hor}= \mu^* (K_M+\Delta) + \sum_{E_i, a_i<-1} (a_i+1) E_i - \Delta'_{vert}.
$$
Suppose now that $(M, \Delta)$ is lc over the generic point of $C$.
Then the restriction of  $\sum_{E_i \ \mbox{\tiny $\mu$-exc.}} (a_i+1) E_i$ 
to a general $(\varphi \circ \mu)$-fibre $F$ is empty. Thus the anti-effective divisor
$
- E := \sum_{E_i, a_i<-1} (a_i+1) E_i - \Delta'_{vert}
$
is vertical. We claim that $E=0$. Assuming this for the time being, let us see how to deduce all the statements:
since  $\sum_{E_i, a_i<-1} (a_i+1) E_i=0$ we know that $(M, \Delta)$ is lc. Moreover any lc centre $W$ surjects onto $C$:
otherwise there exists a vertical divisor $E_i$ with discrepancy $-1$. Thus we have $\Delta'_{vert} \neq 0$, a contradiction.

If $(M, \Delta)$ is klt over the generic point of $C$, it is lc by what precedes. 
Moreover any lc centre would be contained in a $\varphi$-fibre, so it would
give a non-zero component of $\Delta'_{vert}$. However we just proved that  $\Delta'_{vert}=0$.

{\em Proof of the claim.}
Since $- \mu^* (K_{M/C}+\Delta) \equiv (\alpha-1) \mu^* K_{M/C}$ is nef, 
we know that for $H$ an ample Cartier divisor on $M'$ and all $\delta>0$, the divisor
$-\mu^* (K_{M/C}+\Delta)+\delta H$ is ample. 
Thus there exists an effective $\Q$-divisor $B \sim_\Q -\mu^* (K_{M/C}+\Delta) + \delta H$ 
such that the pair $(M', \Delta'_{hor}+B)$ is dlt. Thus we have
$$
K_{M'/C} + \Delta'_{hor} + B \sim_\Q  \mu^* (K_{M/C}+\Delta) - E -  \mu^* (K_{M/C}+\Delta) +\delta H \sim_\Q  - E +\delta H. 
$$
Since $E$ does not dominate $C$, the restriction of $K_{M'/C} + \Delta'_{hor} + B$ to the general 
$(\varphi \circ \mu)$-fibre $F$ 
is numerically equivalent to $\delta H|_F$. In particular for $m \in \N$ sufficiently large and divisible the divisor
$m (K_{M'/C} + \Delta'_{hor} + B)$ is Cartier and the restriction to $F$ has global sections.
The pair  $(M', \Delta'_{hor}+B)$ being lc we know by \cite[Thm.4.13]{Cam04} that  
$(\varphi \circ \mu)_* (m (K_{M'/C} + \Delta'_{hor} + B))$ is a nef vector bundle.
The natural morphism
$$ 
(\varphi \circ \mu)^* (\varphi \circ \mu)_* (m (K_{M'/C} + \Delta'_{hor} + B)) \rightarrow m (K_{M'/C} + \Delta'_{hor} + B)
$$
is not zero, so $m (K_{M'/C} + \Delta'_{hor} + B)$ is pseudoeffective.
Thus we see that $- E +\delta H$ is pseudoeffective. 
Taking the limit $\delta \rightarrow 0$ we deduce that the anti-effective divisor $-E$ is pseudoeffective. This proves the claim.
\end{proof}

{\bf Attribution.}
The proof above is a refinement of
the argument in \cite{LTZZ10}. Note that our argument  can be used
to give a simplified proof of \cite[Thm.]{LTZZ10}.

\begin{corollary} \label{corollaryreductioncurve}
Let $X$ be a projective manifold such that $-K_X$ is nef. Let $\holom{\pi}{X}{T}$ be the Albanese map.
Fix an arbitrary point $t \in T$ and let $C \subset T$ be a smooth curve such that $t \in C$
and for $c \in C$ general, the fibre $\fibre{\pi}{c}$ is smooth.

Then $\fibre{\pi}{C}$ is a normal variety with at most canonical singularities. 
\end{corollary}

\begin{proof}
By \cite[Thm.]{LTZZ10} the fibration $\pi$ is flat, so $\fibre{\pi}{C}$ is Gorenstein and has pure dimension $\dim X-\dim T+1$.
The general fibre of $\fibre{\pi}{C} \rightarrow C$ is irreducible, so $\fibre{\pi}{C}$ is irreducible. 
Since all the $\pi$-fibres are reduced \cite[Thm.]{LTZZ10} we see that 
$\fibre{\pi}{C}$ is smooth in codimension one, hence normal. 
The relative anticanonical divisor $-K_{\fibre{\pi}{C}/C}= -K_X|_{\fibre{\pi}{C}}$ is Cartier and nef.
The general fibre of $\fibre{\pi}{C} \rightarrow C$ is smooth, so $\fibre{\pi}{C}$ is smooth over the generic point of $C$.
Thus the pair $(\fibre{\pi}{C}, 0)$ is klt by Lemma \ref{lemmaimportant}. Since $\fibre{\pi}{C}$ is Gorenstein, it has at most canonical singularities.
\end{proof}

Corollary \ref{corollaryreductioncurve} reduces Conjecture
\ref{conjecturealbanese} to the following problem:

\begin{conjecture} \label{conjecturerelativealbanese}
Let $M$ be a normal projective variety with at most canonical Gorenstein singularities.
Let \holom{\varphi}{M}{C} be a fibration onto a smooth curve $C$ such that
$-K_{M/C}$ is nef. If the general fibre $F$ is smooth, then $\varphi$ is smooth.
If $F$ is smooth and simply connected, then $\varphi$ is locally trivial in the analytic topology.
\end{conjecture} 

\subsection{Running the MMP for fibres of dimension two} \label{subsectionMMP}

In the previous section we reduced Conjecture \ref{conjecturealbanese} to the case 
of a fibration onto a curve such that the total space has canonical singularities. In this section we will
make the stronger assumption that the total space is $\Q$-factorial with terminal singularities.
Moreover we will assume at some point the existence of an effective relative anticanonical divisor.
In Subsection \ref{subsectionmainresult} we will see how to verify these additional conditions.

\begin{setup} \label{setup}
{\rm
Let $X_C$ be a normal $\Q$-factorial projective variety with at most terminal singularities, and 
let \holom{\varphi}{X_C}{C} be a fibration onto a smooth curve $C$. Suppose that the general $\varphi$-fibre is uniruled.
By \cite{BCHM10} we know that $X/C$ is birational to a Mori fibre space, and we denote by
\begin{equation} \label{MMP}
X_C=:X_0 \stackrel{\mu_0}{\dashrightarrow} X_1  \stackrel{\mu_1}{\dashrightarrow}
\ldots  \stackrel{\mu_k}{\dashrightarrow} X_k  
\end{equation}
a MMP over $C$, that is for every $i \in \{0, \ldots, k-1\}$ the map
$\merom{\mu_i}{X_i}{X_{i+1}}$ is either a divisorial Mori contraction of a
$K_{X_i/C}$-negative extremal ray in $\NE{X_i/C}$  or the flip of a small contraction 
of such an extremal ray. Note that all the varieties $X_i$ 
are normal $\Q$-factorial with at most terminal singularities and endowed with a
fibration $X_i \rightarrow C$.
Moreover $X_k \rightarrow C$ admits a Mori fibre space structure \holom{\psi}{X_k}{Y} onto a normal variety
$\holom{\tau}{Y}{C}$. We know 
that $Y$ is $\Q$-factorial \cite[Prop.7.44]{Deb01}, moreover $Y$ has at most klt singularities \cite[Thm.0.2]{Amb05}.
}
\end{setup}

\begin{remark} \label{remarkflip}
If $\mu_i$ is a flip, denote by $\Gamma_i$ the normalisation of its graph and by $\holom{p_i}{\Gamma_i}{X_i}$
and  $\holom{q_i}{\Gamma_i}{X_{i+1}}$ the natural maps. 
By a well-known discrepancy computation \cite[Lemma 9-1-3]{Mat02} we have
\begin{equation} \label{flipdiscrepancy}
p_i^* (-K_{X_i/C}) = q_i^*(-K_{X_{i+1}/C}) - \sum a_{i, j} D_{i,j}
\end{equation}
with $a_{i,j}>0$ where the sum runs over all the $q_i$-exceptional divisors. 
\end{remark}

\begin{remark} \label{remarkdivisorial}
If $\mu_i$ is a divisorial contraction, let $D_i \subset X_i$ be the exceptional divisor. We have 
\begin{equation} \label{divisorialdiscrepancy}
-K_{X_i/C} = \mu_i^*(-K_{X_{i+1}/C}) - \lambda_i D_i
\end{equation}
with $\lambda_i>0$.
\end{remark}

We explained in the introduction that the nefness of $-K_{X_C/C}$ is usually not preserved under the MMP. However
we have the following:

\begin{lemma} \label{lemmaMMPbasic}
In the situation of Setup \ref{setup}, suppose that
\begin{enumerate}[(i)]
\item $-K_{X_C/C}$ is pseudoeffective. Then $-K_{X_i/C}$ is pseudoeffective. 
\item $-K_{X_C/C}$ is nef in codimension one. Then $-K_{X_i/C}$ is nef in codimension one. 
\item $-K_{X_C/C}$ is nef. If $B \subset X_i$ is a curve such that $-K_{X_i/C} \cdot B<0$, then $B$ is 
(a strict transform of) a curve contained in the flipping locus or the image of a divisorial contraction. 
\end{enumerate}
\end{lemma}

\begin{remark} \label{remarkMMPdimtwo}
A $\Q$-Cartier divisor on a surface is nef in codimension one if and only if it is nef.
Thus the lemma shows that for {\em surfaces} the MMP preserves the property that $-K_{X_C/C}$ is nef.
\end{remark}

\begin{proof}  Our proof follows the arguments in \cite[Prop.2.1, Prop.2.2]{PS98}.
Let $\Gamma$ be the normalisation of the graph of the birational map $X \dashrightarrow X_i$, and
denote by $\holom{p}{\Gamma}{X}$ and $\holom{q}{\Gamma}{X_i}$ the natural maps. By
\eqref{flipdiscrepancy} and \eqref{divisorialdiscrepancy} we have
$$
p^* (-K_{X_C/C}) = q^* (-K_{X_i/C}) - D,
$$
with $D$ an effective $q$-exceptional $\Q$-divisor. In particular
if $B \subset X'$ is a curve that is not contained in $q(D)$ and $B_\Gamma \subset \Gamma$ its strict transform, 
then we have
\begin{equation} \label{increase}
p^*(-K_{X_C/C}) \cdot B_\Gamma \leq -K_{X_i/C} \cdot B.
\end{equation}
This proves the statements $(i)$ and $(iii)$. 

Suppose now that 
$-K_{X_C/C}$ is nef in codimension one. Let $S \subset X_i$ be a prime divisor, and
denote by $S_\Gamma \subset \Gamma$ its strict transform. Since $X_i \dashrightarrow X$ does not contract
a divisor, we see that $S_\Gamma$ is not $p$-exceptional. Hence $p^*(-K_{X_C/C})|_{S_\Gamma}$
is pseudoeffective. Thus we see that $q^*(-K_{X_i/C})|_{S_\Gamma}$ is pseudoeffective,
hence $-K_{X_i/C}|_S$ is pseudoeffective.
\end{proof}

We will now restrict ourselves to the case where $\dim X_C-\dim C=2$.
This allows to study
the positivity of $-K_\bullet$ on horizontal curves. 

\begin{definition} \label{definitionnefoverC}
Let $M$ be a projective variety admitting a fibration $\holom{f}{M}{C}$ onto a curve $C$.
A $\Q$-Cartier divisor $L$ on $M$ is nef over $C$ if 
for any curve $B \subset M$ such that $L \cdot B<0$, the image $f(B)$ is a point.
\end{definition}

\begin{remark} \label{remarkflippingloci}
Note that if $\dim X_C-\dim C=2$, the exceptional locus
of a small contraction cannot surject onto $C$, so if $\mu_i$ is the corresponding step of the MMP,
the flipping loci are contained in
fibres of the maps $X_i \rightarrow C$ and $X_{i+1} \rightarrow C$.
In particular if $-K_{X_i/C}$ is nef over $C$, then $-K_{X_{i+1}/C}$ is nef over $C$
by \eqref{increase}. The same holds if the exceptional divisor $D_i$ of a divisorial
contraction does not surject onto $C$.
\end{remark}

\begin{lemma}  \label{lemmanefoverC}
In the situation of Setup \ref{setup}, suppose that $\dim X_C-\dim C=2$ and $-K_{X_i/C}$
is nef over $C$. Then $-K_{X_{i+1}/C}$ is nef over $C$.
\end{lemma}

\begin{proof}
By what precedes it is sufficient to consider the case where 
$\holom{\mu_i}{X_i}{X_{i+1}}$ contracts a divisor $D_i$
onto a curve $C_0$ such that $C_0 \rightarrow C$ is surjective.

We will argue by contradiction, so suppose that $-K_{X_{i+1}/C} \cdot C_0 < 0$. 
Since $-K_{X_i/C}$ is nef over $C$, the restriction of $-K_{X_i/C}$ to 
$D_i \rightarrow C_0$ is nef over $C_0$. 
Since $\mu_i$ is a Mori contraction $-K_{X_i/C}|_{D_i}$ is $\mu_i|_{D_i}$-ample,
so $-K_{X_i/C}|_D$ is nef. By \eqref{divisorialdiscrepancy} one has
$$
-K_{X_i/C} = \mu_i^*(-K_{X_{i+1}/C}) - \lambda_i D_i.
$$
In particular if $B \subset D_i$ is any curve that is not contracted by $\mu_i$, then
$$
-  \lambda_i D_i \cdot B = -K_{X_i/C} \cdot B + K_{X_{i+1}/C} \cdot (\mu_i)_*(B) > 0.
$$
Since $-D_i$ is $\mu_i|_{D_i}$-ample, this shows that $-D_i|_{D_i}$ is positive on all curves. 
Moreover we have
$$
(-  \lambda_i D_i|_{D_i})^2 = (-K_{X_i/C}|_{D_i})^2 + 2 (-K_{X_i/C}) \cdot (\mu_i^* K_{X_i/C}) \cdot D_i > 0,
$$
so by the Nakai-Moishezon criterion the divisor $-D_i|_{D_i}$ is ample. 
Thus we see that $-(K_{X_i/C}+D_i)|_{D_i}$ is ample. Let now 
$\holom{\nu}{\hat D_i}{D_i}$ be the composition of normalisation and minimal resolution, then by the adjunction formula \cite{Rei94}
$$
- K_{\hat D_i/C} = -\nu^* (K_{X_i/C}+D_i)|_{D_i} + F,
$$
with $F$ an effective divisor. Since $D_i \rightarrow C_0$ is generically a $\PP^1$-bundle, the divisor $F$
does not surject onto $C_0$. Thus we see that $-K_{\hat D_i/C}$ is nef over $C$, moreover it is big.
Let $\hat D_i \rightarrow \tilde D_i$ be a MMP over $C$ with $\tilde D_i \rightarrow C$ a Mori fibre space.
Then $-K_{\tilde D_i/C}$ is nef and big, a contradiction to Lemma \ref{lemmanumericaldimensionprojective}.
\end{proof}

Combining Lemma \ref{lemmaMMPbasic} and Lemma \ref{lemmanefoverC} we obtain:

\begin{corollary}  \label{corollarynefoverC}
In the situation of Setup \ref{setup}, suppose that $\dim X_C-\dim C=2$ and $-K_{X_C/C}$ is nef.
Then for every $i \in \{1, \ldots, k\}$ the divisor $-K_{X_i/C}$ is nef in codimension one and nef over $C$.
\end{corollary}
 
Our goal will be to prove that $-K_{X_i/C}$ is actually nef, but this needs some serious preparation.

\begin{lemma} \label{lemmanu1surfaces}
Let $S$ be an integral projective surface admitting a morphism $\holom{f}{S}{C}$ onto a curve $C$.
Let $L$ be a $\Q$-Cartier divisor on $S$ that is pseudoeffective and nef over $C$.
  
Let $S'$ be an integral projective surface admitting a morphism $\holom{f'}{S'}{C}$ onto $C$. 
Let $L'$ be a nef and $f'$-ample $\Q$-Cartier divisor on $S'$ such that $(L')^2=0$.

Suppose that there exists a birational morphism  $\holom{\mu}{S}{S'}$ such that $f' \circ \mu = f$.
Suppose that we have
$$
L \equiv \mu^* L' - E
$$
with $E$ an effective $\Q$-Cartier divisor. Then the following holds:

\begin{enumerate}[(i)]
\item If $L$  is $\mu$-nef, then $L$ is nef and not big. Moreover we have $L \cdot \mu^* L' = 0$.
\item If $L$ is $\mu$-nef and $E$ is $f$-vertical, then $E=0$.
\end{enumerate}
\end{lemma}

Recall that if $L$ is a pseudoeffective $\Q$-divisor on a smooth projective surface,
we can consider its Zariski decomposition
$L = P + N$,
where $P$ is a nef $\Q$-divisor and $N$ an effective $\Q$-divisor such that $P \cdot N=0$.

\begin{proof}
The statement being invariant under normalisation we can suppose without loss of generality that $S$ and $S'$
are normal. For the same reason we can suppose that $S$ is smooth.

Let $L = P + N$  be the Zariski decomposition of $L$, and
let $N = N_{hor} + N_{vert}$ be the decomposition of the 
effective divisor $N$ in its vertical and horizontal components. 

Since $P \equiv \mu^* L' - (E+N)$ and $(L')^2=0$, we have
$$
P \cdot \mu^* L' = -(E+N) \cdot \mu^* L'.
$$
Since $\mu^* L$ and $P$ are nef, the left hand side is non-negative. 
However $-(E+N) \cdot \mu^* L'$ is non-positive, so we get
\begin{equation} \label{orthogonal}
P \cdot \mu^* L' = (E+N) \cdot \mu^* L' = 0.
\end{equation}
In particular for every irreducible component $D \subset N$ we have $D \cdot \mu^* L'=0$.
Since $L'$ is $f'$-ample and $N_{vert} \cdot \mu^* L'=0$ we see that $N_{vert}$ is $\mu$-exceptional.

The divisor $P$ being nef, we have $P^2 \geq 0$
and $(E+N) \cdot P \geq 0$.
However by \eqref{orthogonal} one has
$$
P^2  = [\mu^* L' - (E+N)] \cdot P = - (E+N) \cdot P,
$$
so we get $P^2  = (E+N) \cdot P = 0$.
Yet $(E+N) \cdot P = 0$ and $(E+N) \cdot \mu^* L' = 0$ implies that 
\begin{equation} \label{square}
(E+N)^2 = 0.
\end{equation} 

{\em Proof of  $(i)$.}
Since $N_{vert}$ is $\mu$-exceptional, we know that $L$ is nef on every irreducible
component of $N_{vert}$. Moreover $L$ is nef over $C$, so it is nef 
on every irreducible
component of $N_{hor}$. Thus $L$ is nef on its negative part $N$, hence $N=0$ and $L$ is nef.
In particular we have $L^2=P^2=0$, so $L$ is not big.

{\em Proof of  $(ii)$.} Since $E$ is $f$-vertical and $L'$ is $f'$-ample, 
the equality $E \cdot \mu^* L'=0$ implies that $E$ is $\mu$-exceptional.
Since $N=0$ we know by \eqref{square} that $E^2=0$.
However the intersection matrix of an exceptional divisor is 
negative definite, so we have $E=0$.
\end{proof}

For the next corollary recall that any normal surface is numerically $\Q$-factorial \cite{Sak84}, so
we can define intersection numbers for any Weil divisor. 

\begin{corollary} \label{corollarysurfaces}
Let $Y$ be a normal projective surface 
admitting a fibration $\holom{\tau}{Y}{C}$ onto a smooth curve such that the general fibre is $\PP^1$. 
Suppose that $-K_{Y/C}$ is pseudoeffective and nef over $C$.
Then $-K_{Y/C}$ is nef and $Y \rightarrow C$ is a $\PP^1$-bundle.
\end{corollary}

\begin{proof} 
{\em Step 1. Suppose that $Y$ is smooth.}
We argue by induction on the relative Picard number. If $\rho(Y/C)=1$, then $-K_{Y/C}$ is nef
and $\tau$-ample. Thus $Y \rightarrow C$ is a $\PP^1$-bundle.
If $\rho(Y/C)>1$ there exists a Mori contraction $\holom{\mu}{Y}{Y'}$ over $C$
and by Lemma \ref{lemmaMMPbasic} the divisor $-K_{Y'/C}$ is pseudoeffective and nef over $C$.
By induction $-K_{Y'/C}$ 
is nef and $Y' \rightarrow C$ is a $\PP^1$-bundle. We have $(-K_{Y'/C})^2=0$ and $-K_{Y'/C}$ 
is ample over $C$.  The $\mu$-exceptional divisor $E$
is $\tau$-vertical, so $E=0$ by Lemma \ref{lemmanu1surfaces}, a contradiction.

{\em Step 2. General case.} 
Let $\holom{\nu}{\hat Y}{Y}$ be the minimal resolution, then  we have $K_{\hat Y} = \nu^* K_Y - E$ with $E$ an effective divisor. 
Thus the relative canonical divisor $-K_{\hat Y/C} = - \nu^* K_{Y/C} + E$
is pseudoeffective and nef over $C$. By Step 1 we know that $\hat Y \rightarrow C$
is a $\PP^1$-bundle. Thus $\nu$ is an isomorphism.
\end{proof}

\begin{remark} \label{remarksurfaces}
In the situation of Corollary \ref{corollarysurfaces} we can write $Y \simeq \PP(V)$ with $V$ a rank two vector bundle on $C$.
Since $-K_{Y/C}$ is nef, the vector bundle $V$ is semistable \cite[Prop.2.9]{MP97}.
Moreover the nef cone $\mbox{Nef}(Y) \subset N^1{Y}$ 
and the pseudoeffective cone $\mbox{Pseff}(Y) \subset N^1{Y}$ coincide \cite[Sect.1.5.A]{Laz04a},
they are generated over $\Z$ by $\frac{-1}{2} K_{Y/C}$ and a fibre $F$ of the ruling $Y \rightarrow C$.
Recall that
on any smooth projective surface a Cartier divisor $L$ is generically nef with respect to all polarisations if and only if it is 
pseudoeffective. Thus we see that $L \rightarrow Y$ is generically nef with respect to all polarisations if and only if $L$ 
is nef if and only if
$
L \equiv \frac{-m}{2} K_{Y/C} + n F
$
with $m, n \in \N_0$.
\end{remark}

We will now use Lemma \ref{lemmanu1surfaces} to obtain strong restrictions on the MMP.

\begin{lemma} \label{lemmanu2}
In the situation of Setup \ref{setup}, suppose that $\dim X_C-\dim C=2$.
Let  $\merom{\mu_i}{X_i}{X_{i+1}}$ be a step of the MMP.

Let $S' \subset X_{i+1}$ be an irreducible surface such that
the map $S' \rightarrow C$ is surjective.
Suppose that $-K_{X_{i+1}/C}|_{S'}$ is nef but not big. Suppose moreover that either
\begin{itemize}
\item we have  $-K_{X_{i+1}/C}|_{S'} \equiv 0$; or
\item the divisor  $-K_{X_{i+1}/C}|_{S'}$ is ample on the fibres of $S' \rightarrow C$. 
\end{itemize}
Then the following holds:
\begin{enumerate}[(i)]
\item Suppose that $\mu_i$ is a flip: then $S'$ is disjoint from the flipping locus. 
\item Suppose that $\mu_i$ is divisorial with exceptional divisor $D_i$. 
Then either $\mu_i(D_i)$ is disjoint from $S'$ or $\mu_i(D_i) \subset S'$. 
If $\mu_i(D_i) \subset S'$, the map $\mu_i(D_i)\rightarrow C$ is surjective and 
$-K_{X_{i+1}/C}|_{S'} \not\equiv 0$.
\item Let $S \subset X_i$ be the strict transform of $S'$. Then $-K_{X_i/C}|_{S}$ is nef but not big.
\end{enumerate}
\end{lemma}

\begin{proof}{\em Proof of  (i).}
Arguing by contradiction we suppose that $S'$ is not disjoint from the flipping locus $Z$.
Since $-K_{X_{i+1}/C}|_{S'}$ is nef, the intersection $S' \cap Z$ is non empty and finite.
Using the notation of Remark \ref{remarkflip}, let 
$\Gamma_S \subset \Gamma_i$ be the strict transform of $S'$. 
Restricting \eqref{flipdiscrepancy} to $\Gamma_S$ we obtain
$$
p_i^* (-K_{X_i/C})|_{\Gamma_S} = q_i^*(-K_{X_{i+1}/C})|_{\Gamma_S} 
- (\sum a_{i, j} D_{i, j}) \cap \Gamma_{S}.
$$
The divisor $(\sum a_{i, j} D_{i, j})$ being $\Q$-Cartier, the non-empty intersection
$E:= (\sum a_{i, j} D_{i, j}) \cap \Gamma_S$ is a non-zero $q_i|_{\Gamma_S}$-exceptional effective $\Q$-divisor on $\Gamma_S$.
Since $\Gamma_S$ is not $p_i$-exceptional and surjects onto $C$, it follows
from Corollary \ref{corollarynefoverC} that
$p_i^* (-K_{X_i/C})|_{\Gamma_S}$ is pseudoeffective and nef over $C$,
moreover it is $q_i|_{\Gamma_S}$-nef.
If $-K_{X_{i+1}/C}|_{S'} \equiv 0$ this already gives a contradiction, so suppose now
that $-K_{X_{i+1}/C}|_{S'}$ is ample on the fibres of $S' \rightarrow C$.

Recall that the flipping locus $Z$ is contained in the fibres of $X_{i+1} \rightarrow C$ 
(cf. Remark \ref{remarkflippingloci}), so $E$ is vertical
with respect to the fibration $\Gamma_S \rightarrow C$.
Yet by Lemma \ref{lemmanu1surfaces} applied to the birational map
$\holom{q_i|_{\Gamma_S}}{\Gamma_S}{S'}$ we see that $E=0$, a contradiction.

{\em Proof of  (ii).}
Let $S \subset X_i$ be the strict transform of $S'$, then we have an induced fibration $S \rightarrow C$. 
Using the notation of Remark \ref{remarkdivisorial}  and restricting \eqref{divisorialdiscrepancy} to $S$ we have
\begin{equation} \label{relation}
-K_{X_i/C}|_S = \mu_i^*(-K_{X_{i+1}/C})|_{S'} - \lambda_i (D_i \cap S),
\end{equation}
If $\mu_i(D_i)$ is not disjoint from $S'$,
then $D_i \cap S$ is a non-zero effective divisor. 
Note that the restriction $-K_{X_i/C}|_S$ is pseudoeffective and nef over $C$, moreover 
it is $\mu_i$-ample.
If $-K_{X_{i+1}/C}|_{S'} \equiv 0$ then \eqref{relation} shows that $-K_{X_i/C}|_S$ is anti-effective and not zero,
a contradiction.

Suppose now that $-K_{X_{i+1}/C}|_{S'}$ is ample 
on the fibres of $S' \rightarrow C$. Then we know by Lemma \ref{lemmanu1surfaces} that $D_i \cap S$
is empty unless it has a horizontal component. 
Since $D_i \cap S$ has a horizontal component, the irreducible curve $\mu_i(D_i)$ is contained in $S'$
and surjects onto $C$.

{\em Proof of  (iii).}
If the image of the exceptional (resp. the flipping locus) is disjoint
from $S'$ the statement is trivial. If this is not the case, then by $(i)$ and $(ii)$ 
the contraction $\mu_i$ is divisorial and  
$-K_{X_{i+1}/C}|_{S'}$ is ample on the fibres of $S' \rightarrow C$.
Thus we can apply Lemma \ref{lemmanu1surfaces} to see that 
$-K_{X_i/C}|_{S}$ is nef and not big.
\end{proof}

The next lemma describes the Mori fibre space at the end of the MMP:

\begin{lemma} \label{lemmapreparation}
In the situation of Setup \ref{setup}, suppose that $\dim X_C-\dim C=2$ and that the general $\varphi$-fibre
is rationally connected. Suppose that $Y$ is a surface.

Then $Y$ is smooth, the fibration $\holom{\tau}{Y}{C}$ is a $\PP^1$-bundle and $-K_{Y/C}$ is nef.
Let $\Delta \subset Y$ be the $1$-dimensional part of the $\psi$-singular locus. 
\begin{itemize}
\item If $\Delta \neq 0$, it is a smooth irreducible curve and the map $\Delta \rightarrow C$ is \'etale. We have $\Delta \equiv - \lambda K_{Y/C}$ with $\lambda \in \Q^+$.
Moreover $S':=\fibre{\psi}{\Delta}$ is an irreducible surface such that $-K_{X_k/C}|_{S'}$ is nef but not big. The
restriction $-K_{X_k/C}|_{S'}$ is ample
on the fibres of $S' \rightarrow \Delta$ and if $l \subset S'$ is an irreducible component of a general
fibre of $S' \rightarrow \Delta$, we have $-K_{X_k/C} \cdot l=1$.
\item If $\Delta=0$, then $X_k \rightarrow Y$ is a $\PP^1$-bundle, in particular $X_k$ is smooth.
\end{itemize}
\end{lemma}

\begin{remark*}
If $\Delta \neq 0$ it is -a priori- not clear that $X_k$ is Gorenstein and $X_k \rightarrow Y$ is a conic bundle, 
cp. \cite[\S 12]{MP08} and \cite[p.483]{PS98}.
\end{remark*}

\begin{proof}
By Corollary \ref{corollarynefoverC} we know that $-K_{X_k/C}$ is nef in codimension one and nef over $C$.
By Remark \ref{remarknefcodimone} this implies that $K_{X_k/C}^2$ is a pseudoeffective $1$-cycle.
We claim that $-K_{Y/C}$ is pseudoeffective and nef over $C$. 

{\em Proof of the claim.\footnote{For experts it is not difficult to deduce the claim from general results on the positivity
of direct image sheaves, cf. \cite[Cor.0.2]{BP08} \cite[Lemma 3.24]{a8}.}} 
The fibration $\psi$ does not contract a divisor, so it is equidimensional. 
Since a terminal threefold has at most isolated singularities, there are at most finitely many points 
$Z \subset Y$ such that $(X_k \setminus \fibre{\psi}{Z}) \rightarrow (Y \setminus Z)$
is a conic bundle. Thus we have \cite[4.11]{Miy83}
\begin{equation} \label{miyanishi}
\psi_* (K_{X_k/C}^2) = - (4 K_{Y/C} + \Delta).
\end{equation}
The cycle $K_{X_k/C}^2$ is pseudoeffective, 
so its image $-(4 K_{Y/C} + \Delta)$ is pseudoeffective. This already proves that $-K_{Y/C}$ is pseudoeffective.
We will now follow an argument from \cite[p.482]{PS98}:
let $B \subset Y$ be any irreducible curve that surjects onto $C$. Since $-K_{X_k/C}$
is $\psi$-ample and nef over $C$, the restriction $-K_{X_k/C}|_{\fibre{\psi}{B}}$ is nef.
Thus we see by the projection formula and \eqref{miyanishi} that
\begin{equation} \label{discrim}
0 \leq (-K_{X_k/C})^2 \cdot \fibre{\psi}{B} = - (4 K_{Y/C} + \Delta) \cdot B.
\end{equation}
In particular we have $-K_{Y/C} \cdot B \geq 0$ unless $B \subset \Delta$. Arguing by contradiction we suppose
that there exists an irreducible curve $B \subset \Delta$ such that $-K_{Y/C} \cdot B < 0$.
Since $B \subset \Delta$ the inequality \eqref{discrim} then implies
$$
(K_{Y/C}+B) \cdot B \leq  (K_{Y/C}+\Delta) \cdot B = (4 K_{Y/C} + \Delta) \cdot B + 3  (-K_{Y/C} \cdot B) < 0.
$$
Thus if $\tilde B$ is the normalisation of $B$, the subadjunction formula \cite{Rei94} shows that $\deg K_{\tilde B/C}<0$,
a contradiction to the ramification formula. This proves the claim.

We can now describe the surface $Y$: the general fibre of $X_k \rightarrow C$ is rationally connected, so
the general fibre of $Y \rightarrow C$ is $\PP^1$. Thus we know 
by Corollary \ref{corollarysurfaces} that $Y \rightarrow C$ is a ruled surface and 
$-K_{Y/C}$ is nef.

{\em 1st case. $\Delta \neq \emptyset$.}
Since $\psi$ is a conic bundle in the complement of finitely many points, the general
fibre over a point in $\Delta$ is a reducible conic. It is well-known that if $l \subset S'$ is an irreducible
component of such a reducible conic, then $S' \cdot l=-1$ and $-K_{X_k/C} \cdot l=1$. Using $\rho(X_k/Y)=1$ standard arguments prove that $\Delta$ and $S'$ are irreducible, cf. \cite[Rem.2.3.3]{MP08}.

In the proof of the claim we saw that $-(4 K_{Y/C} + \Delta)$ is pseudoeffective.
Since $-K_{Y/C}$ is nef and $(-K_{Y/C})^2=0$ we obtain
$$
0 \leq -K_{Y/C} \cdot (-4 K_{Y/C} - \Delta) = K_{Y/C} \cdot \Delta \leq 0.
$$ 
Since $Y$ is a ruled surface, the equality $-K_{Y/C} \cdot \Delta=0$ implies that
that $\Delta \equiv - \lambda K_{Y/C}$ with $\lambda \in \Q^+$. In particular $\Delta$ surjects onto $C$ and we have $\Delta^2=0$.
By adjunction we see that $K_{\Delta/C}$ has degree $0$, so $\Delta \rightarrow C$ is \'etale and $\Delta$ is smooth.

Since $\Delta$ surjects onto $C$, the divisor
 $-K_{X_k/C}|_{S'}$ is nef and ample on the fibres
of $S' \rightarrow \Delta$.  Using the projection formula and \eqref{miyanishi} we have
$$
(-K_{X_k/C})^2 \cdot S' = - (4 K_{Y/C} + \Delta) \cdot \Delta = 0,
$$
so $-K_{X_k/C}|_{S'}$ is not big.

{\em 2nd case. $\Delta= \emptyset$.}
The terminal threefold $X_k$ is Cohen-Macaulay and the fibration on the smooth base $Y$ is equidimensional,
so $\psi$ is flat. Moreover $\psi$ has at most finitely many singular fibres. Thus $\psi$ is a $\PP^1$-bundle by
\cite[Thm.2.]{AR12}.
\end{proof}

\begin{proposition} \label{propositionMMPnef}
In the situation of Setup \ref{setup}, suppose that $\dim X_C-\dim C=2$ and that the general $\varphi$-fibre
is rationally connected.
Then $-K_{X_k/C}$ is nef.
\end{proposition}

\begin{proof}
By Corollary \ref{corollarynefoverC} we know that $-K_{X_k/C}$ is nef in codimension one and nef over $C$.
By Lemma \ref{lemmapreparation} we know that
$Y \rightarrow C$ is a ruled surface such that $-K_{Y/C}$ is nef.
Let $\Delta \subset Y$ be the $1$-dimensional part of the $\psi$-singular locus. 

Denote by $\{ B_1, \ldots, B_m\} \subset X_k$
the finite (maybe empty) set of curves such that  $-K_{X_k/C} \cdot B_j<0$. Since $-K_{X_k/C}$ is nef over $C$ and $\psi$-ample,
we see that for all $j \in \{1, \ldots, m\}$, the curve $\psi(B_j)$ is a fibre of the ruling $Y \rightarrow C$.

{\em 1st case: $\Delta \neq \emptyset$.} Let $S' \subset X_k$ be the surface
constructed in Lemma \ref{lemmapreparation}.
We will describe the MMP $X \dashrightarrow X_k$
in a neighbourhood of $S'$.

We set $S_k:=S'$ and for 
every $i \in \{ 0, \ldots, k-1 \}$ we define inductively
$S_{i} \subset X_{i}$ as the strict transform of $S_{i+1} \subset X_{i+1}$.
Consider now the largest $m \in \{ 1, \ldots, k\}$ such that
the surface $S_{m+1}$ is not disjoint from the flipping locus of $\mu_m$ or, if $\mu_m$ is divisorial,
the image $\mu_m(D_m)$ of the exceptional divisor. Since $\mu_k \circ \ldots \mu_{m+1}$
is an isomorphism near $S_{m+1}$ we see that $S_{m+1} \simeq S'$ and
by Lemma \ref{lemmapreparation} the divisor $-K_{X_{m+1}/C}|_{S_{m+1}}$ is nef but not big. 
Moreover $-K_{X_{m+1}/C}|_{S_{m+1}}$ is ample
on the fibres of $S_{m+1} \rightarrow \Delta$. Thus we can apply Lemma \ref{lemmanu2} and see that
$\mu_m$ is divisorial, the curve $\mu_m(D_m)$ is contained in $S_{m+1}$ and surjects onto $C$.
Since $X_{m+1}$ has only finitely many singular points and $\mu_m(D_m)$ is a lci curve in its general point, we see
by \cite[Prop.0.6]{PS98} that $\mu_m$ is {\em generically} the blow-up of the curve $\mu_m(D_m)$.
In particular we have
$$
-K_{X_{m}/C} = - \mu_m^* K_{X_{m+1}/C} - D_m.
$$
Let $l \subset S_{m+1}$ be an irreducible component of a general fibre of $S_{m+1} \rightarrow \Delta$,
and let $l' \subset S_m$ be its strict transform. Since $\mu_m(D_m)$ intersects $l$, we have $D_m \cdot l' \geq 1$.
By Lemma \ref{lemmapreparation} we have 
$$
1 = - K_{X_{m+1}/C} \cdot l = - \mu_m^* K_{X_{m+1}/C} \cdot l', 
$$
so $-K_{X_{m}/C} \cdot l'=0$. By Lemma \ref{lemmanu2} the divisor $-K_{X_{m}/C}|_{S_m}$ is nef.
Since it is numerically trivial on the general fibre of $f: S_m \rightarrow \Delta$, we see that 
$-K_{X_{m}/C}|_{S_m}=f^* H$ with $H$ a nef $\Q$-Cartier divisor on $\Delta$. However by Lemma \ref{lemmanu1surfaces}
we have
$$
(-K_{X_{m}/C}|_{S_m}) \cdot \mu_m^* (-K_{X_{m+1}/C}|_{S_{m+1}})= 0.
$$
Since $-K_{X_{m+1}/C}|_{S_{m+1}}$ is ample on the fibres of $S_{m+1} \rightarrow \Delta$, we see that $H \equiv 0$,
so $-K_{X_{m}/C}|_{S_m} \equiv 0$. Thus we are in the first case of Lemma \ref{lemmanu2}: 
for every $i \in \{ 0, \ldots, m-1 \}$, the MMP
is disjoint from $S_i$.

We will now argue by contradiction and suppose that $-K_{X_k/C}$ is not nef. 
Then the surface $S_k$ meets the curves $B_j$ in finitely many points.
On the one hand we have just seen that the surfaces $S_i$ are disjoint from any flipping locus
of the MMP and if the contraction is divisorial and $S_i$ is not disjoint from $\mu_i(D_i)$, 
then $\mu_i(D_i)$ surjects onto $C$. 
On the other hand we know by Lemma \ref{lemmaMMPbasic}, (iii) that
$B_j$ is contained in a flipping locus or the image of an exceptional 
divisor. Thus we have $B_j \cap S_k = \emptyset$, a contradiction.

{\em 2nd case: $\Delta = \emptyset$.} 
By Lemma \ref{lemmapreparation} the variety $X_k$ is smooth and
$X_k \rightarrow Y$ is a $\PP^1$-bundle. Using the Grothendieck-Riemann-Roch formula \cite[App.A, Thm.5.3]{Har77}
we see that
$$
c_1(\psi_* (\omega^*_{X_k/Y}))=0, \ c_2(\psi_* (\omega^*_{X_k/Y}))=\frac{1}{2} K_{X_k/Y}^3.
$$
Since $\psi_* (\omega^*_{X_k/C}) \simeq \psi_* (\omega^*_{X_k/Y}) \otimes \omega^*_{Y/C}$ and
$K_{Y/C}^2=0$ we deduce that
$$
c_1(\psi_* (\omega^*_{X_k/C}))= - 3 K_{Y/C}, \ c_2(\psi_* (\omega^*_{X_k/C}))=\frac{1}{2} K_{X_k/C}^3.
$$
Let $A \subset Y$ be a general hyperplane section, then $\fibre{\psi}{A} \rightarrow A$ is a $\PP^1$-bundle
and $-K_{X_k/C}|_{\fibre{\psi}{A}}$ is nef, since $\fibre{\psi}{A} \cap B_j$ is a finite set for every $j$.
In particular the direct image $\psi_* (\omega^*_{X_k/C})|_A$ is nef. Thus $\psi_* (\omega^*_{X_k/C})$
is generically nef for any polarisation $A$ and by \cite[Thm.6.1]{Miy87} one has
$$
c_2(\psi_*( \omega^*_{X_k/C})) \geq 0.
$$
We claim that $K_{X_k/C}^3 \leq 0$, by what precedes this implies $c_2(\psi_*( \omega^*_{X_k/C})) = 0$.

{\em Proof of the claim.} Recall that $K_{X_k/C}^2$ is a pseudoeffective cycle and
$\psi_* (K_{X_k/C}^2) = - 4 K_{Y/C}$. Let $(K_n)_{n \in \N}$ be a sequence of effective $1$-cycles with rational coefficients 
converging in $N_1(X_k)$ to $K_{X_k/C}^2$. Then we can write
$$
K_n = \sum_{j=1}^m \eta_{j, n} B_j + R_n,
$$
where $R_n$ is an effective $1$-cycle with rational coefficients such that $-K_{X_k/C} \cdot R_n \geq 0$. 
If $H$ is an ample divisor on $X_k$, the degrees $H \cdot (\eta_{j, n} B_j)$ and $H \cdot R_n$ are bounded
for large $n \in \N$ by $(H \cdot K_{X_k/C}^2)+1$. Thus, up to replacing $K_n$ by some subsequence, we can suppose
that the sequences  $\eta_{j, n}$ and $R_n$ converge. Thus we have
$$
K_{X_k/C}^2 = \sum_{j=1}^m \eta_{j, \infty} B_j + R_\infty,
$$
where $R_\infty$ is a pseudoeffective cycle such that 
$-K_{X_k/C} \cdot R_\infty \geq 0$. Pushing down to $Y$ we have
$$
- 4 K_{Y/C} = \sum_{i=j}^m \eta_{j, \infty} \psi_*(B_j) + \psi_*(R_\infty).
$$
Recall now that $K_{Y/C}^2=0$ and $-K_{Y/C} \cdot \psi_*(B_j)>0$ for all $j$. Then the preceding equation
shows that $\eta_{j, \infty}=0$ for all $j$, hence we get $K_{X_k/C}^2 = R_\infty$ and
$$
- K_{X_k/C}^3 = -K_{X_k/C} \cdot R_\infty \geq 0.
$$
This proves the claim.

{\em Conclusion.}
We will now prove that $\psi_* (\omega^*_{X_k/C})$ is a nef vector bundle. Since the natural morphism
$$
\psi^* \psi_* (\omega^*_{X_k/C}) \rightarrow \omega^*_{X_k/C}
$$
is surjective, this proves that $-K_{X_k/C}$ is nef. If $\psi_* (\omega^*_{X_k/C})$ is stable for some
polarisation $A$ the property
$$
c_1^2(\psi_* (\omega^*_{X_k/C}))= - 3 K_{Y/C}^2=0, \ c_2(\psi_* (\omega^*_{X_k/C}))=0
$$
implies  that $\psi_* (\omega^*_{X_k/C})$ is projectively flat with nef determinant \cite[Cor.3]{BS94}, hence nef.
We already know that $V:=\psi_* (\omega^*_{X_k/C})$ is generically nef for any polarisation $A$ on $Y$. Thus
if $V \rightarrow Q$ is any torsion-free quotient sheaf, then $Q$ is generically nef for any polarisation $A$ on $Y$.
By Remark \ref{remarksurfaces} we then have
$$
\det Q = \frac{-m}{2} K_{Y/C} + n F
$$
with $m, n \in \N_0$. 
Suppose now that $V$ is not stable with respect to the polarisation $\frac{-1}{2} K_{Y/C} + \frac{1}{8} F$.
Then there exists a stable reflexive subsheaf $\sF \subset V$ such that the quotient $Q:=V/\sF$ is torsion-free and
the slope $\mu(Q)$ is less or equal than the slope $\mu(V)$.
An elementary computation shows that
$\det Q=\frac{-m}{2} K_{Y/C}$ with $m \leq 2 \rk Q$. In particular we have $\det \sF= \frac{-(6-m)}{2} K_{Y/C}$.
Since $c_2(Q) \geq 0$ by \cite[Thm.6.1]{Miy87} and $c_2(\sF) \geq 0$ by the Bogomolov-Miyaoka-Yau inequality
we see that $c_2(\sF) = 0$ and $c_2(Q) = 0$. In particular $\sF$ is projectively flat with nef determinant \cite[Cor.3]{BS94},
hence nef. The same holds for $Q$ if it is stable. If $Q$ is not stable we easily prove that it is an extension
of two line bundles $L_1$ and $L_2$ which are non-negative multiples of $\frac{-m}{2} K_{Y/C}$, in particular $Q$ is nef.
Thus $V$ is an extension of nef vector bundles, hence nef.
\end{proof}

\begin{remark*}
The proof of the second case $\Delta = 0$ is tedious and rather ad-hoc. If we could suppose the existence of
a curve $C_0 \subset Y$ such that $-K_{Y/C} \cdot C_0=0$ we could argue as in the first case.
Unfortunately the curve $C_0$ does not always exist, cf. Remark \ref{remarkmumford}.
\end{remark*}

\begin{proposition} \label{propositionMMPtotal}
In the situation of Setup \ref{setup}, suppose that $\dim X_C-\dim C=2$ and that the general $\varphi$-fibre
is rationally connected. 
Suppose also that there exists an effective divisor $A_0 \subset X_C$ such that
$A_0 \equiv -mK_{X_C/C}$ for some $m \in \N$.
Then every $\mu_i$ is a divisorial contraction onto some \'etale multisection,
and $-K_{X_i/C}$ is nef for all $i \in \{ 0, \ldots, k\}$. 
If $X_C$ is Gorenstein, the fibration $X_C \rightarrow C$ is locally trivial in the analytic topology.
\end{proposition}

\begin{proof}
Note first that $-K_{X_k/C}$ is nef: if $\dim Y=2$ this was shown in Proposition \ref{propositionMMPnef},
if $\dim Y=1$ the Mori fibre space $\psi$ and the fibration $X_k \rightarrow C$ identify, 
so $-K_{X_k/C}$ is relatively ample and nef over $C$, hence nef.
Let $F$ be a general fibre of $X_k \rightarrow C$. 

{\em Step 1. Description of the MMP.}
We denote by $M \in \mbox{Pic}^0 C$ a divisor such that
$$
A_0 \in H^0(X_C, -mK_{X_C/C} + \varphi^* M).
$$
For every $i \in \{0, \ldots, k\}$ we denote by $M_i$ the pull-back of $M$ to $X_i$ via the natural map $X_i \rightarrow C$.
Setting inductively $A_{i+1}=(\mu_i)_* A_i$ for every $i \in \{0, \ldots, k-1\}$
we have
$$
A_i \in H^0(X_i, -mK_{X_i/C} + M_i) \qquad \forall \ i \in \{0, \ldots, k\}.
$$
The divisor $A_i$ being effective we see that $-K_{X_i/C}$ is nef if and only if 
its restriction to every irreducible component of $A_i$ is nef.
Note that if the contraction $\mu_i$ is divisorial with exceptional divisor $D_i$, we have
$$
H^0(X_{i+1}, (\mu_i)_*(\sO_{X_i}(-mK_{X_i/C} + M_i))) = H^0(X_{i+1}, (-mK_{X_{i+1}/C} + M_{i+1}) \otimes \sJ),
$$
where $\sJ$ is an ideal sheaf whose cosupport is $\mu_i(D_i)$. In particular $A_{i+1}$ contains $\mu_i(D_i)$.
Since $-K_{X_k/C}$ is nef, the contraction $\mu_k$ is divisorial.

{\em 1st case. Suppose that $K_{F}^2=0$.}  By Lemma \ref{lemmapreparation}
and \eqref{miyanishi} we have
$\psi_* (K_{X_k/C}^2) \equiv - m K_{Y/C}$
with $m \geq 0$. Restricting to the general fibre $F$ the condition $K_F^2=0$ implies that $m=0$.
Thus the pseudoeffective cycle $K_{X_k/C}^2$
is numerically equivalent to $\mu l$ where $l$ is a general $\psi$-fibre. 
Thus we see that $0=(-K_{X_k/C})^3 = 2 \mu$.
Hence we have $K_{X_k/C}^2 \equiv 0$, in particular the restriction of $-K_{X_k/C}$ to any component
of $A_k$ is numerically trivial. If the MMP $X \dashrightarrow X_k$ is not an isomorphism, then
$\mu_k(D_k)$ is not disjoint from $A_k$, a contradiction to 
Lemma \ref{lemmanu2}(ii).

{\em 2nd case. Suppose that $K_{F}^2>0$.}
In this case we can apply Corollary \ref{corollaryMFS} to see 
that $X_k \rightarrow C$ is locally trivial in the analytic topology, in particular $X_k$ is smooth. 

Suppose for the moment that $-K_{X_i/C}$ is nef and relatively big for some $i \in \{1, \ldots, k\}$. 
We claim that  the divisor $A_i$ is a union of irreducible components
$A_i = \sum b_{i,l} A_{i,l}$ such that for every $l$ the natural map $A_{i,l} \rightarrow C$ is surjective
and $-K_{X_i/C}|_{A_{i,l}}$ is either numerically trivial or nef and relatively ample, but not big.

{\em Proof of the claim.}
Since $-K_{X_i/C}$ is nef but not big, the restriction $-K_{X_i/C}|_{A_{i,l}}$ is not big.
Since $(\varphi_i)_*(\omega^{\otimes -m}_{X_i/C}) \otimes M$ is numerically flat we can see as in the proof
of Proposition \ref{propositionloctriv} that $A_i \rightarrow C$ is locally trivial. 
In particular all the irreducible components surject onto $C$ and
if $-K_{X_i/C}|_{A_{i,l}}$ is relatively big for the fibration $A_{i,l} \rightarrow C$, it is relatively ample.
Thus we are left to show that if $-K_{X_i/C}$ is numerically trivial on the fibres of $A_{i,l} \rightarrow C$,
then its restriction to $A_{i,l}$ is numerically trivial. Note that in this case  $A_{i,l}$
is contracted by the morphism to the relative anticanonical model 
$$
\nu_i: X_i \rightarrow X_i' \subset \PP((\varphi_i)_*(\omega^{\otimes -d}_{X_i/C}))=:\PP(V_i)
$$  
with $d \gg 0$,  and the image
of $\nu_i(A_{i,l})$ is an irreducible component of $(X_i')_{sing}$. Since $X_i' \rightarrow C$ is 
locally trivial, the curve $\nu_i(A_{i,l})$ is an \'etale multisection and a connected component of $(X_i')_{sing}$.
After \'etale base change we can suppose that $\nu_i(A_{i,l})$ is a section.
If $\sI_{X_i'}$ is the ideal of $X_i'$ in $\PP(V_i)$ the direct image
$(\varphi_i)_* (\sI_{X_i'} \otimes \sO_{\PP(V_i)}(e))$ is numerically flat for $e \gg 0$ (cf. the proof of Theorem \ref{theoremmain}),
so $(\varphi_i)_* (\sI_{(X_i')_{sing}} \otimes \sO_{\PP(V_i)}(e))$ is also numerically flat. 
In particular the section $\nu_i(A_{i,l})$ corresponds to a numerically trivial quotient of $V_i$, hence
$$
0 = \sO_{\PP(V_i)}(e) \cdot (X_i')_{sing} = - e d K_{X_i'/C} \cdot (X_i')_{sing}.
$$
Since $\nu_i$ is crepant this proves the claim.

We will now prove by descending induction that $-K_{X_i/C}$ is nef and relatively big for all $i \in \{1, \ldots, k\}$.  
This is clear for $i=k$, so suppose that it holds for $i+1$.
Since $-K_{X_{i+1}/C}$ is nef, the contraction $\mu_i$ is divisorial
with exceptional divisor $D_i$ and $\mu_i(D_i)$ is contained in $A_{i+1}$.
By the claim the irreducible components $A_{i+1,l}$ satisfy the conditions of Lemma
\ref{lemmanu2}. Thus $\mu_i(D_i)$ is a curve surjecting onto $C$
and the divisor $-K_{X_i/C}$ is nef on the strict transforms
of all the irreducible components $A_{i+1,l}$.
This already proves that  $-K_{X_{i}/C}|_{A_{i}}$ is nef, unless
$A_i$ has one irreducible component more than $A_{i+1}$, the exceptional divisor $D_i$. 
Clearly $-K_{X_{i}/C}|_{D_i}$ is relatively ample with respect to $D_i \rightarrow \mu_i(D_i)$.
Yet $\mu_i(D_i)$ surjects onto $C$, so the restriction $-K_{X_{i}/C}|_{D_i}$ is
nef over $\mu_i(D_i)$. This proves that $-K_{X_{i}/C}|_{D_i}$ is nef, hence $-K_{X_{i}/C}|_{A_{i}}$ is nef.
Since $-K_{X_{i}/C}$ is nef we know by \cite[Prop.4.11]{PS98} that $\mu_i$ 
is the blow-up along an \'etale (multi-)section.

{\em Step 2. $\varphi$ is locally trivial.} By what precedes the first step of MMP is a divisorial
contraction $\holom{\mu_0}{X_0}{X_1}$ contracting a divisor $D_0$ onto an \'etale 
multisection\footnote{The case of a trivial MMP can be excluded as follows: 
consider the Mori fibre space $\psi: X_C = X_k \rightarrow Y$.
Since $X_C$ is Gorenstein, the fibration $\psi$ is a conic bundle \cite[Thm.7]{Cut88}. 
The discriminant locus $\Delta$ is smooth
by Lemma \ref{lemmapreparation}, so all the fibres over $\Delta$ are reducible conics \cite[Prop.1.8.3)]{Sar82}.
Thus the associated two-to-one cover $\tilde \Delta \rightarrow \Delta$ is \'etale, hence $\tilde \Delta \rightarrow C$
is \'etale by Lemma \ref{lemmapreparation}. Arguing as \cite[Prop.0.4, Rem.0.5]{PS98} we see that 
$X_C \times_C \tilde \Delta \rightarrow \tilde \Delta$ admits a Mori contraction that blows down exactly one (-1)-curve
in every fibre.}.
In particular $-K_{X_1/C}$ is nef and $\varphi_1$-big where $\holom{\varphi_1}{X_1}{C}$ is the natural fibration.
By Corollary \ref{corollaryMFS} the variety $X_1$ is smooth, so $X_0$ is smooth.
Moreover $-K_{X_0/C} - \mu_0^* K_{X_1/C}$ is nef and $\varphi$-big,
hence for $m \in \N$ the direct image $\varphi_*(\omega_{X_0/C}^{\otimes -m} \otimes \mu_0^* \omega_{X_1/C}^{\otimes -m})$
is nef \cite{Kol86}. The inclusion $(\mu_0)_* (\omega_{X_0/C}^{\otimes -m}) \hookrightarrow \omega_{X_1/C}^{\otimes -m}$ 
yields an inclusion 
$$
\varphi_*(\omega_{X_0/C}^{\otimes -m} \otimes \mu_0^* \omega_{X_1/C}^{\otimes -m})
\hookrightarrow
(\varphi_1)_* (\omega_{X_1/C}^{\otimes -2m}).
$$
The sheaf $(\varphi_1)_* (\omega_{X_1/C}^{\otimes -2m})$ is numerically flat for all $m \gg 0$ by
Proposition \ref{propositionkltdirectimage}, so 
$\varphi_*(\omega_{X_0/C}^{\otimes -m} \otimes \mu_0^* \omega_{X_1/C}^{\otimes -m})$ is also numerically flat
for all $m \gg 0$.
By the relative base-point free theorem the natural map 
$$
\varphi^* \varphi_*(\omega_{X_0/C}^{\otimes -m} \otimes \mu_0^* \omega_{X_1/C}^{\otimes -m})
\rightarrow 
\omega_{X_0/C}^{\otimes -m} \otimes \mu_0^* \omega_{X_1/C}^{\otimes -m}
$$
is surjective for all $m \gg 0$, so we obtain a birational morphism
$\holom{\mu}{X_0}{X_0'}$ onto a normal projective variety $\varphi': X_0' \rightarrow C$
embedded in $\varphi': \PP(E_m) \rightarrow C$
where $E_m :=  \varphi_*(\omega_{X_0/C}^{\otimes -m} \otimes \mu_0^* \omega_{X_1/C}^{\otimes -m})$
for some fixed $m \gg 0$.  We can now argue as in the proof of Theorem \ref{theoremmain}:
denoting by $\sI_{X_0'} \subset \sO_{\PP(E_m)}$ the ideal sheaf of $X_0' \subset \PP(E_m)$,
we have for
every $d \gg 0$ an exact sequence
$$
0 \rightarrow (\varphi')_{*}(\sI_{X_0'}\otimes \mathcal{O}_{\mathbb{P}(E_{m})}(d)) 
\rightarrow S^d E_m \rightarrow \varphi_*(\omega_{X_0/C}^{\otimes -dm} \otimes \mu_0^* \omega_{X_1/C}^{\otimes -dm}) \rightarrow 0.
$$
Thus $(\varphi')_{*}(\sI_{X_0'}\otimes \mathcal{O}_{\mathbb{P}(E_{m})}(d)) $ 
is numerically flat, so $\varphi': X_0' \rightarrow C$ is locally trivial with fibre $F'$ by Proposition \ref{propositionloctriv}.

The Cartier divisor $-K_{X_0/C}$ is $\mu$-trivial, 
so we have $K_{X_0/C} = \mu^* K_{X_0'/C}$ \cite[Thm.3.24]{KM98}. Thus $X_0 \rightarrow X_0'$ is a crepant resolution of
the normal projective variety $X_0'$, 
and a smooth $\varphi$-fibre $F$ is the minimal resolution of the general $\varphi'$-fibre $F'$.
In particular $F$ is unique up to isomorphism, arguing exactly as in Step 2 of the proof of Theorem \ref{theoremmain}
we see that $X_C \rightarrow C$ is locally trivial with fibre $F$.
\end{proof}

\subsection{Main result} \label{subsectionmainresult}

In this section we will prove Theorem \ref{theoremmaintwo}. Proposition \ref{propositionMMPtotal}
obviously settles the main part, however we proved the statement
under a nonvanishing condition which is not satisfied in general (cf. Remark \ref{remarkmumford}).
We will now show that these properties hold if we start with a fibration onto a torus.

\begin{lemma} \label{lemmanonvanishing}
Let $X$ be a projective manifold such that $-K_X$ is nef.
Let $\pi: X \rightarrow A$ be the Albanese fibration, 
and suppose that $-K_F$ is nef and abundant\footnote{Cf. \cite{Fuj11} for the relevant definitions.} for the general $\pi$-fibre $F$.
Then there exists an effective divisor $A \subset X$ such that
$A \equiv -mK_{X}$ for some $m \in \N$.
\end{lemma}

\begin{proof}
Since $-K_F$ is nef and abundant we know by the relative version of Kawamata's 
theorem \cite[Thm.1.1]{Fuj11}
that $-K_X$ is $\pi$-semiample,  so
for every sufficiently divisible $m \gg 0$ the natural map
$$
\pi^{*}\pi_{*}(\omega_X^{\otimes -m})\rightarrow \omega_X^{\otimes -m}
$$
is surjective. Thus $-mK_X$ induces a morphism
$\holom{\psi}{X}{Y}$
onto a normal projective variety $\holom{\tau}{Y}{A}$ 
such that $-K_X \sim_\Q \psi^* H$
with $H$ a nef and $\tau$-ample Cartier divisor. 
Since $\pi$ is equidimensional \cite{LTZZ10}, the fibration $\tau$ is equidimensional.
By \cite[Thm.0.2]{Amb05} there exists a boundary divisor $\Delta_Y$ on $Y$ such that the pair $(Y, \Delta_Y)$ is klt
and $H \sim_\Q -(K_Y+\Delta_Y)$. In particular the variety $Y$ is Cohen-Macaulay, so the equidimensional fibration
$\tau$ is flat. Thus we can apply Proposition \ref{propositionkltdirectimage} to see that for sufficiently divisible $m \gg 0$, 
the direct image sheaf
$$
\pi_* (\omega_X^{- \otimes m}) \simeq \tau_* (\sO_Y(-m(K_Y+\Delta_Y)))
$$
is a numerically flat vector bundle. By \cite[Thm.1.18]{DPS94} there exists a subbundle $F \subset \pi_* (\omega_X^{- \otimes m})$
such that $F$ is given by a unitary representation $\pi_1(A) \rightarrow U(\rk F)$. The group $\pi_1(A)$
is abelian, so the representation splits, i.e. $F$ is a direct sum of numerically trivial line bundles. In particular
there exists a $M \in \mbox{Pic}^0 A$ such that
$H^0(A, M \otimes F) \neq 0$.
\end{proof}

\begin{lemma} \label{lemmaGRR}
Let \holom{f}{M}{C} be a fibration from a normal $\Q$-factorial threefold with
at most Gorenstein terminal singularities onto a curve $C$ such that $-K_{M/C}$
is nef. Suppose that the general fibre $F$ is rationally connected and $K_F^2=0$.
\begin{enumerate}[(i)]
\item Then we have
$c_1(f_!(\omega^*_{M/C}))=0$.
\item If $h^0(F, -K_F)=1$, there exists an effective divisor $A \subset M$ 
such that $A \equiv -K_{M/C}$.
\end{enumerate}
\end{lemma}

\begin{remark} \label{remarkmumford}
Let $C$ be a curve of genus at least two, and let $U$ be a rank two bundle of degree $0$ on $C$
such that all the symmetric powers $S^m U$ are stable (such vector bundles have been constructed by Mumford).
Set $M:=\PP(U)$, then $-K_{M/C}$ is nef, but not numerically equivalent to any effective $\Q$-divisor.
\end{remark}

\begin{proof}
{\em Proof of (i)} By the Grothendieck-Riemann-Roch formula 
\cite[Thm.15.2]{Ful98}\footnote{The statement in \cite{Ful98} is only for a smooth total space,
but if $\holom{\mu}{M'}{M}$ is a resolution of singularities one checks easily that
$ch(-K_{M/C}) td(T_M)=ch(-\mu^* K_{M/C}) td(T_{M'})$. Thus the formula holds since
we can apply \cite[Thm.15.2]{Ful98} to $f \circ \mu$.}
we have
\[
td(T_C) ch(f_! (\omega_{M/C}^*)) = f_* (ch(-K_{M/C}) td(T_M)).
\]
We will prove that the degree $3$ component of $ch(-K_{M/C}) td(T_M)$ is equal to $1-g$,
which by the formula above implies the statement.

Since $K_{M/C}^3=0$ and $K_F^2=0$ we have
\begin{equation} \label{equationeasy}
K_{M/C}^2 \cdot K_M = 0, \qquad K_{M/C} \cdot K_M^2 = 0.
\end{equation}
The Chern character of $-K_{M/C}$ is 
$ch(-K_{M/C}) = 1 - K_{M/C} + \frac{1}{2} K_{M/C}^2$, and the Todd class is
\[
td(T_M) = 1- \frac{K_M}{2} + \frac{K_M^2+c_2(M)}{12} + \chi(M, \sO_M).
\]
Thus the degree $3$ components of $ch(-K_{M/C}) td(T_M)$ are given by
\begin{equation} \label{equationeasy2}
\chi(M, \sO_M) - K_{M/C} \cdot \frac{K_M^2+c_2(M)}{12} + \frac{1}{4} K_{M/C}^2 \cdot K_M.
\end{equation}
Since we have $\chi(M, \sO_M)=-\frac{K_M c_2(M)}{24}$ and $f^* K_C \cdot c_2(M) =
(2g-2) c_2(F)$ and $c_2(F)=12$ we obtain
$
- K_{M/C} \cdot \frac{c_2(M)}{12} = 2 \chi(M, \sO_M) + 2g-2$. 
Using \eqref{equationeasy} the formula \eqref{equationeasy2} simplifies to
$3 \chi(M, \sO_M) + 2g-2$.
The general $f$-fibre is rationally connected, so we have 
$\chi(M, \sO_M)=\chi(C, \sO_C)=1-g$.

{\em Proof of (ii)} Note that for a general fibre $F$ we have 
$h^1(F, -K_F)=h^2(F, -K_F)=0$, so 
$R^j f_* (\omega_{M/C}^*)$ is a torsion sheaf for $j \geq 1$. If $F_0$ is an arbitrary fibre, then by
Serre duality $h^2(F_0, -K_{F_0})=h^0(F_0, 2 K_{F_0})=0$ since $K_{F_0}$ is by hypothesis antinef
and not trivial.
Thus we have $R^2 f_* (\omega_{M/C}^*)=0$ and
$$
f_! (\omega_{M/C}^*) = f_* (\omega_{M/C}^*) - R^1 f_* (\omega_{M/C}^*).
$$
Since $R^1 f_* (\omega_{M/C}^*)$ is a torsion sheaf, statement $(i)$ implies that
$c_1 (f_* (\omega_{M/C}^*)) \geq 0$.
Thus $f_* (\omega^*_{M/C})$ is a line bundle of non-negative degree, so there exists a
numerically trivial line bundle $L$ on $C$ such that
$H^0(C, f_* (\omega^*_{M/C}) \otimes L) \neq 0$.
\end{proof}

\begin{proposition} \label{propositioneasy}
Let $X$ be a projective manifold such that  $-K_X$ is nef, and 
let \holom{\pi}{X}{T} be the Albanese map. Suppose that $\dim T=\dim X-2$ and the general
$\pi$-fibre $F$ is uniruled but not rationally connected. Then there exists a finite \'etale cover $X' \rightarrow X$
such that $q(X')=\dim X-1$. Moreover the fibration $\pi$ is smooth.
\end{proposition}

\begin{proof}
Let $X \dashrightarrow Y$ be a model of the MRC-fibration \cite{Deb01} such that $Y$ is smooth. Then $Y$ is not uniruled,
and we denote by $K_Y=P+N$ the divisorial Zariski decomposition. By \cite[Main Thm.]{Zha05} the positive part $P$ is 
zero\footnote{The statement in \cite[Main Thm.]{Zha05} is only $\kappa(Y)=0$, but the proof consists in showing that $P=0$.}.
By \cite[Cor.3.4]{Dru11} the variety $Y/T$ has a good minimal model $Y'/T$.
By \cite[Prop.8.3]{Kaw85} there exists a finite \'etale cover $\tilde T \rightarrow T$ such that $\tilde T \times_T Y'$ is a torus.
Since the irregularity is invariant under the MMP, we see that $q(\tilde T \times_T Y)=\dim X-1$,
thus $q(\tilde T \times_T X)=\dim X-1$. 
This proves the first statement.  

Let now $X' \rightarrow X$ be an \'etale cover such that $q(X')=\dim X-1$, and let $X' \rightarrow T'$ be the Albanese map.
By Corollary \ref{corollarymain} we know that $X' \rightarrow T'$ is a $\PP^1$-bundle. In particular the reduction of every $\pi$-fibre
is a $\PP^1$-bundle $F_0$ over an elliptic curve $E$. Let $\holom{\psi}{X}{Y}$ be a Mori contraction over $T$,
then $\psi$ is a $\PP^1$-bundle, and the (reductions of) fibres of $\tau: Y \rightarrow T$ are elliptic curves.
Thus we have $K_Y \equiv 0$, by the Beauville-Bogomolov decomposition the fibration $\tau$ is smooth.
Hence $\pi= \tau \circ \psi$ is smooth. 
\end{proof}

\begin{remark} \label{remarkeasy}
Proposition \ref{propositioneasy} also holds if $X$ is a compact K\"ahler threefold: using \cite{Pau12} we see that
the base $Y$ of the MRC fibration has $\kappa(Y)=0$. Since $Y$ is a surface we can run the MMP, in the surface case Kawamata's result \cite{Kaw85} follows from the Kodaira-Enriques classification.
\end{remark}

\begin{proof}[Proof of Theorem \ref{theoremmaintwo}]
Let $F$ be a general $\pi$-fibre. Using Beauville-Bogomolov we easily exclude the case where $F$
is not uniruled. If $F$ is uniruled but not rationally connected we conclude by Proposition \ref{propositioneasy}. Suppose now that $F$ is rationally connected. 
If $K_F^2>0$ we conclude by Theorem \ref{theoremmain}, so suppose $K_F^2=0$.

We fix an arbitrary $t \in T$ 
and $C \subset T$ a general smooth curve such that $t \in C$.
By Corollary \ref{corollaryreductioncurve} the preimage
$\fibre{\pi}{C}$ is a normal variety with at most canonical singularities. The divisor 
$-K_{\fibre{\pi}{C}/C}$ is Cartier and nef, so if $X_C \rightarrow \fibre{\pi}{C}$ is a terminal $\Q$-factorial model 
\cite[Thm.6.23, Thm.6.25]{KM98} the  divisor 
$-K_{X_C/C}$ is Cartier and nef. By Proposition \ref{propositionMMPtotal} the fibration $X_C \rightarrow C$ is locally trivial if we prove
that there exists an effective $\Q$-divisor $A_0$ such that $A_0 \equiv -mK_{X_C/C}$ with $m \in \N$.
If $-K_F$ is not abundant this holds by Lemma \ref{lemmaGRR},b).
If $-K_F$ is nef and abundant we know by Lemma \ref{lemmanonvanishing} that there exists an effective divisor $A$ on $X$
such that $A \equiv -mK_X$. Since $C \subset T$ is general, the restriction $A|_{\fibre{\pi}{C}}$ is an effective divisor 
that is numerically equivalent to $-mK_{\fibre{\pi}{C}/C}$. The pull-back of this divisor to $X_C$ then gives $A_0$.

Thus we know that $X_C \rightarrow C$ is locally trivial. In particular any curve in a fibre of $X_C \rightarrow C$ deforms into a general fibre,
hence $X_C \rightarrow \fibre{\pi}{C}$ is an isomorphism.
Thus $X_C \simeq \fibre{\pi}{C} \rightarrow C$ is locally trivial.
\end{proof}

The proof of Theorem \ref{theoremmaintwo} would be much simpler if we could classify
manifolds such that $-K_X$ is nef but not semiample. The following example shows that this is non-trivial
even for threefolds, there by correcting \cite[p.498]{PS98} and \cite[p.600]{Pet12}.

\begin{example}
Let $C$ be an elliptic curve, and let $L_0 \in \mbox{Pic}^0 C$ be a line bundle of degree $0$ that is not torsion.
Set $L_1 := L_2 := \sO_C$ and $L_3 := L_0^{\otimes -2}$. Then $V:= \oplus_{i=0}^3 L_i$ is a numerically flat
vector bundle of rank four, and we denote by $\holom{\psi}{\PP(V)}{C}$ the projectivisation. The vector bundle $S^3 V$ contains a subvector bundle 
$$
(L_0^{\otimes 2} \otimes L_3) \oplus L_1^{\otimes 3} \oplus L_2^{\otimes 3} \simeq \sO_C^{\oplus 3},
$$ 
so $\sO_{\PP(V)}(3)$ has global sections corresponding fibrewise to the degree three monomials
$
X_0^2 X_3, \ X_1^3, \ X_2^3$.
The polynomial $X_0^2 X_3 + X_1^3 + X_2^3$ defines a cubic surface in $\PP^3$ that is normal and has 
a unique singular point in $(0:0:0:1)$, this point is of type $D_4$ \cite[Case C]{BW79}.
Thus if we denote by $X \subset \PP(V)$ the hypersurface defined by the global section
of  $\sO_{\PP(V)}(3)$ corresponding to this polynomial, we see that $X$ is normal with at most canonical singularities.
Moreover we have
$$
\omega_X^* \simeq (\psi^* L_0 \otimes \sO_{\PP(V)}(1))|_X,
$$
so $-K_X$ is nef. The singular locus of $X$ is the curve $C_0$ defined fibrewise by $X_0=X_1=X_2=0$, so it corresponds to the quotient $V \rightarrow L_3.$
Since $L_3 = L_0^{\otimes -2}$ we see that
$\omega^*_X|_{C_0} \simeq L_0^*.$ 
The line bundle $L_0$ is not torsion, so we obtain that $C_0 \subset \Bs |-mK_X|$ for all $m \in \N$. In particular $-K_X$ is not semiample.

Let now $X' \rightarrow X$ be a terminal model obtained by taking fibrewise the minimal resolution, then
$X'$ is smooth and $-K_{X'}$ is nef and not semiample. One checks easily that $X'$ is not a product, even after finite \'etale cover.
\end{example}

\begin{proof}[Proof of Corollary \ref{corollarykaehler}]
If $X$ is projective we conclude by Corollary \ref{corollarymain} and Theorem \ref{theoremmaintwo}.
Thus we are left to deal with the case where $X$ is not projective and $q(X)=1$. Then the general fibre $F$ of
$\holom{\pi}{X}{T}$ is not rationally connected, since otherwise $H^2(X, \sO_X)=0$. If $F$ is uniruled we apply 
Remark \ref{remarkeasy}. If $F$ is not uniruled, the canonical bundle $K_X$ is pseudoeffective \cite{Bru06}.
Thus $K_X \equiv 0$ and we conclude by Beauville-Bogomolov.
\end{proof}

\begin{appendix}

\section{A Hovanskii-Teissier inequality} \label{appendixinequality}

For the convenience of reader, we give the proof of the Hovanskii-Teissier concavity inequality for arbitrary compact K\"ahler manifolds,
which was first proved in \cite{Gro}. The proof here is a direct consequence of \cite[Thm A, C]{DN06}.

\begin{proposition} \label{propositionHT}
Let $(X,\omega_{X})$ be a compact K\"ahler manifold of dimension $n$, 
and let $\alpha$, $\beta$ be two nef classes. For every $i,j, k, s \in \N$ we
have
\begin{equation} \label{equationseven}
\int_{X}(\alpha^{i}\wedge\beta^{j}\wedge\omega_{X}^{n-i-j})
\end{equation}
$$\geq 
(\int_{X}\alpha^{i-k}\wedge\beta^{j+k}\wedge\omega_{X}^{n-i-j})^{\frac{s}{k+s}}
\cdot(\int_{X}\alpha^{i+s}\wedge\beta^{j-s}\wedge\omega_{X}^{n-i-j})^{\frac{k}{k+s}}.
$$
\end{proposition}

\begin{proof}
Let $\omega_{1},\cdots, \omega_{n-2}$ be $n-2$ arbitrary K\"ahler classes. 
Thanks to \cite[Thm.A]{DN06}, 
the bilinear form on $H^{1,1}(X)$
$$Q([\lambda], [\mu])=\int_{X}\lambda\wedge\mu\wedge\omega_{1}\wedge\cdots\wedge\omega_{n-2}
\qquad\lambda, \mu\in H^{1,1}(X)$$
is of signature $(1, h^{1,1}-1)$.
Since $\alpha$, $\beta$ are nef classes,
the function $f(t)=Q (\alpha+t\beta, \alpha+t\beta)$ is indefinite on $\mathbb{R}$ 
if and only if $\alpha$ and $\beta$ are linearly independent.
Therefore 
\begin{equation}\label{equationaddapen}
\int_{X}(\alpha\wedge\beta\wedge\omega_{1}\wedge\cdots\wedge\omega_{n-2})\geq 
(\int_{X}\alpha^{2}\wedge\omega_{1}\wedge\cdots\wedge\omega_{n-2})^{\frac{1}{2}}
\cdot(\int_{X}\beta^{2}\wedge\omega_{1}\wedge\cdots\wedge\omega_{n-2})^{\frac{1}{2}}.
\end{equation}

If we let $\omega_{1}, \cdots ,\omega_{i-1}$ tend to $\alpha$,
let $\omega_{i},\cdots,\omega_{i+j-2}$ tend to $\beta$
and take $\omega_{i+j-1}=\cdots=\omega_{n-2}=\omega_{X}$ in \eqref{equationaddapen},
we have
$$\int_{X}(\alpha^{i}\wedge\beta^{j}\wedge\omega_{X}^{n-i-j})\geq 
(\int_{X}\alpha^{i-1}\wedge\beta^{j+1}\wedge\omega_{X}^{n-i-j})^{\frac{1}{2}}
\cdot(\int_{X}\alpha^{i+1}\wedge\beta^{j-1}\wedge\omega_{X}^{n-i-j})^{\frac{1}{2}} .$$
Then \eqref{equationseven} is an easy consequence of the above inequality.
\end{proof}

\begin{remark}\label{corHT}
It is easy to see that the equality holds in \eqref{equationaddapen} if and only if $\alpha$ and $\beta$ are colinear. 
\end{remark}

\end{appendix}

\def\cprime{$'$}

\end{document}